\documentclass[a4paper,11pt,english,reqno]{amsart}

\usepackage[english]{babel}
\usepackage{amssymb}
\usepackage{color}
\usepackage{appendix}
\usepackage{bm, upgreek}
\usepackage{soul}
\usepackage{hyperref}
\usepackage{mathtools}
\usepackage{dsfont}
\numberwithin{equation}{section}

\usepackage[letterpaper,top=2cm,bottom=2cm,left=2cm,right=2cm,marginparwidth=1.75cm]{geometry}
\renewcommand{\leq}{\leqslant}
\renewcommand{\geq}{\geqslant}
\newcommand{\cbar}{\overline{C}}
\newcommand{\cstar}{\overline{C}_{*}}
\newcommand{\N}{\mathbb{N}}
\newcommand{\R}{\mathbb{R}}
\newcommand{\I}{\mathcal{I}}
\newcommand{\T}{\mathcal{T}}
\newcommand{\W}{\mathcal{W}}
\newcommand{\E}{\mathbb{E}}
\newcommand{\UT}{\textup{UT}(\delta)}
\newcommand{\UTH}{\textup{UT}_{\textup{H}}(\delta)}
\newcommand{\uH}{\textup{H}}
\renewcommand{\P}{\mathbb{P}}
\newcommand{\G}{\mathbb{G}}
\newcommand{\glow}{\textup{G}_{\textup{low}}}
\newcommand{\ghigh}{\textup{G}_{\textup{High}}}
\newcommand{\tu}[1]{\textup{#1}}

\DeclareMathOperator{\dBer}{Ber}
\newcommand{\set}[1]{[\![#1]\!]}
\def\mathbi#1{\textbf{\em #1}}
\usepackage{amsthm}
\newtheorem{theorem}{Theorem}[section]
\newtheorem{proposition}[theorem]{Proposition}
\newtheorem{definition}[theorem]{Definition}
\newtheorem{remark}[theorem]{Remark}
\newtheorem{conjecture}[theorem]{Conjecture}
\newtheorem{lemma}[theorem]{Lemma}

\usepackage{amsmath}
\usepackage{graphicx}

\definecolor{purple}{rgb}{0.9,0,0.8}

\definecolor{gray}{rgb}{0.7,0.7,0.7}

\begin{document} 
	\title{Upper tail bounds for irregular graphs}

	\author{Anirban Basak}
	\address{Anirban Basak, International Centre for Theoretical Sciences, Tata Institute of Fundamental Research, Bangalore, India}
	\email{anirban.basak@icts.res.in}
	\author{Shaibal Karmakar}
	\address{Shaibal Karmakar, International Centre for Theoretical Sciences, Tata Institute of Fundamental Research, Bangalore, India}
	\email{shaibal.karmakar@icts.res.in}

	\maketitle
	\begin{abstract}
		We consider the upper tail large deviations of subgraph counts for irregular graphs $\uH$ in $\mathbb{G}(n,p)$, the sparse Erd\H{o}s-R\'enyi graph on $n$ vertices with edge connectivity probability $p \in (0,1)$. For $n^{-1/\Delta} \ll p \ll 1$, where $\Delta$ is the maximum degree of $\uH$, we derive the upper tail large deviations for any irregular graph $\uH$. On the other hand, we show that for $p$ such that $1 \ll n^{v_\uH} p^{e_\uH} \ll (\log n)^{\alpha^{*}_{\uH}/\left(\alpha^{*}_{\uH}-1\right)}$, where $v_\uH$ and $e_\uH$ denote the number of vertices and edges of $\uH$, and $\alpha^*_\uH$ denotes the fractional independence number, the upper tail large deviations of the number of unlabelled copies of $\uH$ in $\G(n,p)$ is given by that of a sequence of Poisson random variables with diverging mean, for any strictly balanced graph $\uH$. Restricting to the $r$-armed star graph we further prove a localized behavior in the intermediate range of $p$ (left open by the above two results) and show that the mean-field approximation is asymptotically tight for the logarithm of the upper tail probability. This work further identifies the typical structures of $\mathbb{G}(n,p)$ conditioned on upper tail rare events in the localized regime.
	\end{abstract}

	\section{Introduction and main results}
	The classical large deviation theory traditionally deals with large deviations of linear functions of independent and identically distributed (i.i.d.)~random variables. Probably the simplest non-trivial problem that falls outside realm of the classical theory is the large deviations of triangle counts in an Erd\H{o}s-R\'enyi random graph, to be denoted by $\G(n,p)$, with vertex set $\set{n}:=\{1,2,\ldots, n\}$ and edge connectivity probability $p=p_n \in (0,1)$, where each pair of vertices is connected with a probability $p$ independently of every other pair.  In the last fifteen years there have been immense interest in studying large deviations of subgraph (in particular triangle) counts in Erd\H{o}s-R\'enyi random graphs, both in the dense regime, i.e.~$p \asymp 1$ (see the monograph \cite{chatterjeemono} and the references therein), and in the sparse regime, i.e.~$p \ll 1$, e.g.~see \cite{augeri, BGLZ, chatterjeedembo, cookdembo, CDP, eldan} (we refer the reader to Section \ref{sec:notation} for all the notation).

	Recently Harel, Mousset, and Samotij \cite{HMS} showed that the speed and the rate function of upper tail of cliques (i.e.~complete subgraphs) on $r$ vertices in $\G(n,p)$ undergoes a transition at the threshold $p^{({r-1})/{2}} \asymp n^{-1} (\log n)^{{1}/({r-2})}$. For $p^{({r-1})/{2}} \ll n^{-1} (\log n)^{{1}/({r-2})}$ (and $p^{(r-1)/2} \gg n^{-1}$ so that the expected number of cliques on $r$ vertices diverges) the speed and the rate function are given by those for a sequence of Poisson random variables with diverging means, and thus that regime of $p$ is rightly termed as the {\em Poisson regime}. In contrast, for $  n^{-1} (\log n)^{{1}/({r-2})} \ll p^{({r-1})/{2}}  \ll 1$ the large deviation event is primarily due to the presence of some localized structures (of microscopic size) in the graph, and therefore that regime is termed as the {\em localized regime}.

	It was conjectured in (an earlier version of) the work \cite{HMS} that the transition between the localized and the Poisson regimes, for the upper tail of {\em all} connected regular subgraph counts, should occur at the threshold $p^{\Delta/2} \asymp n^{-1} (\log n)^{1/(v_\uH-2)}$, where $\Delta$ is the common degree of the regular graph $\uH$ in context and $v_\uH$ denotes the number of vertices in $\uH$. In \cite{HMS} the Poisson behavior was proved in the predicted Poisson regime, and the localized behavior for $p$ such that $ n^{-1/2-o(1)} \le p^{\Delta/2} \ll 1$ (and under the additional assumption that $\uH$ is non bipartite in the extended range $n^{-1} (\log n)^{\Delta v_\uH^2} \ll p \ll 1$).   
	The more recent work \cite{ABRB} extends the localized behavior for {\em all} regular graph in the sub regime left open by \cite{HMS}. Therefore, \cite{HMS} and \cite{ABRB} together settle the problem of upper tail large deviations of regular subgraph counts in sparse Erd\H{o}s-R\'enyi graphs (albeit a couple of boundary cases). Let us also mention in passing that the speed and the rate function in the localized regime turns out to be the solution of an appropriate mean field variational problem. See Section \ref{sec:mean-field} below for a further discussion on this.

	In this short article we investigate whether for the upper tail of irregular subgraph (i.e.~there exists at least two vertices with unequal degrees) counts in sparse Erd\H{o}s-R\'enyi graphs a transition threshold between the localized and the Poisson regimes exists, and whether in the localized regime the large deviation probability can be expressed as the solution of some mean field variational problem (similar to the one mentioned in the regular setting). 


	\subsection{Localized regime I} We split the localized regime in two sub regimes. The rationale is that in the first sub regime the optimizer for the associated variational problem is a `planted' one, namely the `planted hub' in the irregular case. In contrast, in the second sub regime the optimizer is predicted to be  `non-planted' (meaning it can take value outside $\{p,1\}$; the rationale behind the nomenclature will be clear in Section \ref{sec:mean-field}), as evidenced through Theorem \ref{thm:mean-field}. 
For a graph $\uH$ we write $e_\uH$ to denote the number of its edges.  We use the notation $N(\uH, \tu{G})$ to denote the number of {\em labelled} copies of $\uH$ in $\tu{G}$. In the graph theory literature, this quantity is often denoted by $\tu{inj(H, G)}$ (see \cite{Lovasz}), the number of injective homomorphisms, or by $\tu{Emb(H,G)}$ (see \cite{HMS}), the number of embeddings of $\uH$ into $\tu{G}$.  Moreover, $N(\uH, \tu{G})$ is equal to number of (unlabelled) copies of $\uH$ in $\tu{G}$ multiplied by a factor of $|\tu{Aut}(H)|$, the number of automorphisms of $\uH$.  Throughout the paper we will use the shorthand\\
	\begin{equation*}
		\UTH\coloneqq\{N(\uH, \G(n,p))\geq (1+\delta)n^{v_{\uH}}p^{e_{\uH}}\}.
	\end{equation*}
	
	\begin{theorem}\label{thm:loc-1}
		Let $\Delta\geq 2$ be an integer, and ${\uH}$ be a connected, irregular graph with maximum degree $\Delta$. For every fixed $\delta>0$ and all $p\in(0,1)$ satisfying $n^{-1/\Delta}\ll p\ll1$,
		\begin{equation*}
			\lim_{n\rightarrow \infty} \frac{-\log\P\left(\UTH\right)}{n^{2}p^{\Delta}\log(1/p)} = \theta_{{\uH}^*},
		\end{equation*}
		where $\theta_{{\uH}^{*}} =\theta_{\uH^*}(\delta)$ is the unique positive solution to the equation $P_{{\uH}^{*}}(\theta)=1+\delta$ and $P_{{\uH}^{*}}$ is the independence polynomial (see \cite[Definition 1.1]{BGLZ}) of the graph ${\uH}^{*}$ with ${\uH}^{*}$ being the induced subgraph of ${\uH}$ on all vertices whose degree in $\uH$ is $\Delta$.
	\end{theorem}

	Upper tail large deviations of general subgraph counts in $r$-uniform Erd\H{o}s-R\'enyi random hypergraphs have been studied in \cite{CDP}. Their result when translated to Erd\H{o}s-R\'enyi random graphs ($r=2$) yield the upper tail large deviations for subgraph counts for any irregular subgraph $\uH$, with maximum degree $\Delta$, in the regime $n^{-1/\widetilde \Delta} \ll p \ll 1$, where $\widetilde \Delta = \Delta$ if $H$ is a star graph, and $\widetilde \Delta = \Delta +1$ otherwise. 
	Thus Theorem \ref{thm:loc-1} improves that result.  
	
	With some additional work we further obtain the typical structure of $\G(n,p)$ conditioned on the atypical upper tail event $\UTH$. To state the result we need a few notation. For any vertex $u$ in a graph $\tu{G}$ we write $\deg_{\tu{G}}(u)$ to denotes its degree. For $U_1$ and $U_2$ subsets of the vertex set of a graph $\tu{G}$ we write $\tu{G}([U_1,U_2])$ to denote the bipartite subgraph induced by edges with one endpoint in $U_1$ and the other in $U_2$. Moreover, for any graph $\tu{G}'$ we write $e(\tu{G}')$ to denote its number of edges. 
	
	\begin{theorem}
		\label{cond_struc:loc-1}
		Consider the same setup as in Theorem \ref{thm:loc-1}.  Then for any fixed $\chi\in (0,1)$ there exists some constant $c_\chi >0$ such that
		\begin{equation*}
			\P \left(\tu{Hub}_\chi(\delta)\ \big|\ \UTH\right) \ge 1 - \exp(-c_\chi r_{n,p}),
		\end{equation*} 
		or all large $n$, where $r_{n,p} \coloneqq n^2 p^\Delta \log(1/p)$ and
		\begin{equation}
			\label{Hub:defn}
			\begin{aligned}
				\tu{Hub}_\chi(\delta)\coloneqq \big\{\G(n,p)\tu{ contains a set } &U\subseteq \set{n} \tu{ s.t. } \tu{deg}_{\G(n,p)}(u)\geq (1-\chi)n\ \forall u\in U \tu{ and } \\
				&e({\G(n,p)}\left[U,\set{n}\setminus U\right])\geq (1-\chi)\theta_{\uH^{*}}n^{2}p^{\Delta}\big\}.
			\end{aligned}
		\end{equation}
		
	\end{theorem}
	
	Observe that Theorem \ref{cond_struc:loc-1} shows that conditioned on the upper tail event $\UTH$ the Erd\H{o}s-R\'enyi random graph contains a subgraph that is `close' to a complete bipartite graph (often termed as a hub in the literature) with probability approaching one. Note that Theorem \ref{cond_struc:loc-1} immediately implies Theorem \ref{thm:loc-1}.
	
	Although the primary focus of this paper is to study the upper tail problem for irregular graphs, the arguments employed in the proof of Theorem \ref{cond_struc:loc-1} can be extended to obtain the typical structure of $\G(n,p)$ conditioned on the upper tail event of a regular graph. We include that result in this article for its potential future usage. To state it we need to introduce some notation.
	For $U$ a subset of the vertex set of $\tu{G}$, we write $\tu{G}[U]$ to denote the subgraph induced by the vertices in $U$. 
	\begin{theorem}
		\label{thm: cond struc regular graph}
		
		Let $\Delta\geq 2 $ be an integer, and $\uH$ be a connected $\Delta$-regular graph. Fix $\delta>0$. Assume $p\in (0,1)$ satisfy $p \ll 1$ and $np^{\Delta/2}\gg (\log n)^{1/(v_{\uH}-2)}$. For any fixed $\chi\in (0,1)$ there exists some constant $c_\chi > 0$ such that the following hold.
		
		\begin{itemize}
			\item[(a)] If $np^{\Delta}\gg 1$ then we have
			\begin{equation}\label{eq:clique-hub}
				\P\left( \tu{Clique}_{\chi}(\delta)\cup \tu{Hub}_{\chi}(\delta)\big| \UTH\right) \geq 1 - \exp(-c_\chi r_{n,p}),
			\end{equation}
			\text{ for all large $n$}, where
			\begin{equation}\label{eq:clique-or-hub}
				\begin{aligned}
					\tu{Clique}_{\chi}(\delta) \coloneqq \big\{\G(n,p) \tu{ contains a set }&U\subseteq \set{n} \tu{ of size at least }(1-\chi)\delta^{1/v_{\uH}}np^{\Delta/2}\tu{ and }\\
					&\min_{u\in U} \tu{deg}_{\G(n,p)[U]} (u)\geq (1-\chi) |U|\big\}.
				\end{aligned}
			\end{equation}
			Moreover, there exists some $\delta_0(\tu{H})$ such that, 
			for all large $n$ 
			we have
			\[
			\P\left(  \tu{Hub}_{\chi}(\delta)\big| \UTH\right) {\mathds{1}}_{\delta > \delta_0(\tu{H})}+ \P\left(  \tu{Clique}_{\chi}(\delta)\big| \UTH\right){\mathds{1}}_{\delta < \delta_0(\tu{H})}  \leq \exp(-c_\chi r_{n,p}),
			\]

			\item[(b)] If $np^{\Delta}\ll 1$ and $\tu{H}$ is non-bipartite then for all large $n$ we have
			\begin{equation*}
				 \P\left( \tu{Clique}_{\chi}(\delta)\big| \UTH\right) \geq 1 - \exp(-c_\chi r_{n,p}).
			\end{equation*}
			%
			%
			
		\end{itemize}

	\end{theorem}
	
	Theorem \ref{thm: cond struc regular graph} was proved in \cite{HMS} for $\uH=K_r$, the clique on $r$ vertices. In the context of homomorphism densities,  analogs of Theorems \ref{cond_struc:loc-1} and \ref{thm: cond struc regular graph} were obtained in \cite{CD24} (their result allows one to condition on the intersection of the upper tail events of multiple subgraphs) in the regime $p \gg n^{-1/(\Delta+1)}$, where $\Delta$ is the maximal degree of the subgraphs under consideration. 
	Theorem \ref{thm: cond struc regular graph} is proved in Appendix \ref{sec:proof-regular}.
			%
	
	\subsection{Poisson regime} For any graph $\tu{J}$ 
	a fractional independent set is a map $\alpha:V(\tu{J})\rightarrow [0,1]$ such that $\alpha(u)+\alpha(v)\leq 1$ whenever $uv\in E(\tu{J})$. The {\em fractional independence number} of $\tu{J}$, denoted by $\alpha^{*}_{J}$, is the largest value of $\sum_{v\in V(\tu{J})}\alpha(v)$ among all fractional independent sets $\alpha$ in $\tu{J}$.
	The next result obtains the upper tail large deviations of the number of copies of $\uH$ for any {{\em strictly balanced graph}} (i.e.~$e_{\uH}/v_{H} > e_{\tu{J}}/v_{\tu{J}}$ for any proper subgraph $\tu{J}$ of $\uH$). Since any regular graph is strictly balanced, it is a  generalization of \cite[Theorem 1.6]{HMS}. 
	Moreover, we believe that the range of $p$ and the choice of $\uH$ in the theorem below are optimal.
	\begin{theorem}\label{thm:poi}
		Let $\uH$ be a strictly balanced graph. 
		Fix $\delta>0$. For $p\in (0,1)$ satisfying $1 \ll n^{v_{\uH}}p^{e_{\uH}}\ll (\log n)^{\alpha^{*}_{\uH}/\left(\alpha^{*}_{\uH}-1\right)}$, we have
		\begin{equation*}
			\lim_{n\rightarrow \infty} \frac{-\log \P\left(N(\uH, \G(n,p))\geq (1+\delta)n^{v_{\uH}}p^{e_{\uH}}\right)}{ n^{v_{\uH}}p^{e_{\uH}}/ \tu{Aut}({\uH})} = (1+\delta)\log(1+\delta)-\delta,
		\end{equation*}
		where $\tu{Aut}(\uH)$ is the number of bijective maps $\phi: V(\uH)\rightarrow V(\uH)$ such that $(u,v)\in E(\uH)$ if and only if $(\phi(u), \phi(v))\in E(\uH)$.
	\end{theorem}

	The assumption on $\uH$ in Theorem \ref{thm:poi} is necessary to get the Poisson behavior (even the speed) near the appearance threshold, e.g.~in the regime $1 \ll n^{v_\uH} p^{e_\uH} \ll (\log n)^{c_\uH}$, where $c_\uH >0$ is some constant (see the counterexample  in \cite[Theorem 1]{counterexample}). On the other hand, if $n^{v_\uH} p^{e_\uH} \asymp 1$ then it is well known that the random $N(\uH,\G(n,p))$ is asymptotically Poisson (see \cite[Theorem 1]{KR83}) for any strictly balanced graph $\uH$. For balanced graphs the asymptotic behavior of $N(\uH, \G(n,p))$ at the appearance threshold is more involved and the limit is not necessarily a Poisson random variable.

	\subsection{Localized regime II} Theorems \ref{thm:loc-1} and \ref{thm:poi} leave open the range $n^{v_\uH} p^{e_\uH} \gg (\log n)^{\alpha^{*}_{\uH}/\left(\alpha^{*}_{\uH}-1\right)} $ and $p \ll n^{-1/\Delta}$. The following result yields large deviations in that intermediate regime for $H=K_{1,r}$, the $r$-armed star graph, for $r \ge 2$. 
	\begin{theorem}\label{thm:loc-2}
		Let $r\geq 2$. Fix $\delta>0$. Then for $p\in(0,1)$ satisfying $p\lesssim n^{-1/r}$ and $n^{r+1}p^{r}\gg (\log n)^{r/(r-1)}$, we have
		\begin{equation}
			\label{eq:star loc regime 2 result}
			\lim_{n \to \infty}  \frac{-\log\P\left(N(K_{1,r}, \G(n,p))\geq (1+\delta)n^{r+1}p^r\right)}{n^{1+1/r} p \log n} = \begin{cases}
				\frac{1}{r}\delta^{1/r}\quad &\text{if $np^{r}\rightarrow 0$},\\
				\frac{1}{r\rho^{1/r} }\left(\lfloor\delta \rho\rfloor+\{\delta\rho\}^{1/r}\right) &\text{if $np^{r}\rightarrow \rho\in (0,\infty)$}.
			\end{cases}
		\end{equation}
	\end{theorem}

	\begin{remark}
		The recent work \cite{AS25} also studies the upper tail probability of the number of copies of $K_{1,r}$.
		 A key to their proofs in \cite{AS25} is the observation that $N(K_{1,r}, \G(n,p))$ can be expressed as a sum of some function of the degrees of vertices in $\G(n,p)$. This essentially allowed them to study a functional of the degree sequence, which in turn is approximated by a collection of i.i.d. binomial random variables. Consequently, the problem reduces to analyzing the asymptotics of the corresponding log-probabilities (see \cite[Lemma~4.2]{AS25}). This approach also allowed them to consider the regime $n^{r+1}p^r =\Theta ((\log n)^{r/(r-1)})$ revealing a mixture of both Poisson and localization behaviour in that regime.


		However, the approach of \cite{AS25} does not seem to be amenable to address upper tail probabilities of subgraph counts for irregular graphs $\uH \ne K_{1,r}$, even in the regime $p \gg n^{-1/\Delta}$, as considered in Theorem \ref{thm:loc-1}. 
		Our proof follows a different route which is based on the machineries introduced in \cite{HMS, ABRB}. Furthermore, unlike Theorem \ref{cond_struc:star loc-2}, \cite{AS25} does not provide a description of the typical structure of the random graph conditioned on the presence of excess copies of $K_{1,r}$.
		
	\end{remark}
			%
	The next result derives the typical structure of $\G(n,p)$ conditioned on the upper tail rare event of $K_{1,r}$ in the setting of Theorem \ref{thm:loc-2}.
	\begin{theorem}
		\label{cond_struc:star loc-2}
		Let $r\geq 2$. Fix $\delta>0$. Assume $p\in (0,1)$, satisfies $n^{r+1}p^{r}\gg \left(\log n\right)^{r/(r-1)}$. For any fixed $\chi\in (0,1)$, there exists $c_\chi >0$ such that the following hold.
		\begin{itemize}
			\item[(a)] If $np^{r}\ll 1$ then 
			\begin{equation*}
				\P\left(\tu{High-degree}_{\chi}(\delta)\big| \tu{UT}_{K_{1,r}}(\delta)\right)  \geq 1 - \exp(-c_\chi \widetilde{r}_{n,p}),
			\end{equation*}
			for all large $n$, where $\widetilde{r}_{n,p}\coloneqq n^{1+1/r} p \log n$ and
			\begin{equation*}
				\tu{High-degree}_{\chi}(\delta)\coloneqq \left\{\max_{v \in \set{n}} \deg_{\G(n,p)}(v) \geq (1-\chi)\delta^{1/r} n^{1+1/r}p \right\}.
			\end{equation*}
			\item[(b)] If $np^{r}\rightarrow \rho\in (0,\infty)$ then 
			\begin{equation*}
				\P\left(\widetilde{\tu{Hub}}_{\chi}(\delta)\big| \tu{UT}_{K_{1,r}}(\delta)\right) \geq 1-\exp(-c_\chi \widetilde{r}_{n,p}),
			\end{equation*}
			for all large $n$, where 
			\begin{equation*}
				\begin{aligned}
					\widetilde{\tu{Hub}}_{\chi}(\delta)\coloneqq &\bigg\{
					\exists U\subseteq \set{n}\tu{ of size }\lfloor\delta\rho\rfloor\tu{ such that } \min_{u \in U} \deg_{\G(n,p)}(u) \geq (1-\chi)n\\& \qquad \qquad \qquad  \tu{ and }  \max_{v \in \set{n} \setminus U} \deg_{\G(n,p)}(v) \geq  (1-\chi)\{\delta\rho\}^{1/r}n  \bigg\}.
				\end{aligned}
			\end{equation*}
		\end{itemize}
	\end{theorem}
	
	Comparing Theorem \ref{cond_struc:loc-1} with Theorem \ref{cond_struc:star loc-2} we see that the difference between the regimes $p \gg n^{-1/r}$ and $p \asymp n^{-1/r}$ is that in the former regime $\G(n,p)$, conditioned on the rare event $\tu{UT}_{K_{1,r}}(\delta)$, typically contains an almost complete bipartite graph of appropriate size, while in the latter regime it may additionally contain a vertex of high degree, depending on whether $\delta \rho \notin \N$ or not. For $p \ll n^{-1/r}$ (and $p \gg n^{-(1+1/r)} (\log n)^{1/(r-1)}$) the first scenario ceases to appear and in that regime $\G(n,p)$, conditioned on $\tu{UT}_{K_{1,r}}(\delta)$, typically contains a vertex of high degree.
	
	\begin{remark}
		In light of Theorem \ref{cond_struc:star loc-2}(a) and \cite[Theorems 1.1 and Proposition 4.4]{AS25}, we expect that 
there exists a threshold $c_{r,\delta}\in (0,\infty)$ such that 
for $c<c_{r,\delta}$, conditioned on $\tu{UT}_{K_{1,r}}(\delta)$,  the random graph $\G(n,p)$ would typically consist of approximately $\delta n^{r+1}p^{r}$ labelled copies of $K_{1,r}$ that are almost edge disjoint. 
On the other hand, if $c>c_{r,\delta}$, then the typical structure should exhibit a mixture behaviour: the event $\tu{High-degree}_\chi(\delta^{\prime})$ holds for some $\delta^{\prime}\in (0,\infty)$ and  $(\delta-\delta^{\prime})n^{r+1}p^{r}$ labelled copies are almost edge disjoint. 

	\end{remark}

	\subsection{Na\"ive mean-field approximation}\label{sec:mean-field}
	For a function $h: \{0,1\}^N \mapsto \R$ and the uniform measure $\mu$ on ${\sf C}_N:= \{0,1\}^N$, the na\"ive mean-field variational problem provides an approximation for the {\em $\log$-partition function} $\log Z_h$, where $Z_h:=\int \exp(h) d\mu$.  
	There have been several works, in different settings, attempting to find sufficient conditions on $h(\cdot)$, e.g.~appropriate {\em low-complexity} conditions on the (discrete) gradient of $h(\cdot)$, such that the mean-field approximation is asymptotically tight (see \cite{augeri, austin, BM, chatterjeedembo, eldan, Yan}). 
	A heuristic computation shows that the mean-field variational problem for the logarithm of the upper tail probability  $\mu_p(f \ge (1+\delta) \E_{\mu_p}[f])$, where {$f : \{0,1\}^N \mapsto \R$} is some `nice' function, $\delta >0$, and $\mu_p$ is the product of $N$ i.i.d.~$\dBer(p)$ measures, is formulated as below:
	\begin{equation}\label{eq:mf-vp3}
		\Psi_{p,f}(\delta):= \inf \left\{ I_p({\bm \xi}): {\bm \xi} \in [0,1]^N \text{ and } \E_{\mu_{\bm \xi}}[f] \ge (1+\delta) \E_{\mu_p}[f]\right\},
	\end{equation}
	\[
	I_p({\bm \xi}):= \sum_{\upalpha=1}^N I_p(\xi_\alpha), \quad {\bm \xi}:= (\xi_1, \xi_2, \ldots, \xi_N), \quad \text{and} \quad I_p(x):= x \log \frac{x}{p} + (1-x) \log \frac{1-x}{1-p} \text{ for } x \in [0,1],
	\]
	with the convention $0 \log 0 = 0$ and  the probability measure $\mu_{\bm \xi}:=\otimes_{\upalpha=1}^N \dBer(\xi_\upalpha)$.

	Since, upon setting $N= {n \choose 2}$, identifying each of the possible ${n\choose2}$ edges to $\set{N}$, and letting $f$ to be the number of copies of any subgraph in a graph on $n$ vertices, the upper tail of $N(\uH, \G(n,p))$ falls under the framework described above and it has been of interest whether the log-probability of such events is well approximated by $\Psi_{p,\uH}(\cdot)\coloneqq \Psi_{p,N(\uH, \cdot)}(\cdot)$.

	In \cite{BGLZ}, the authors computed a related variant of \eqref{eq:mf-vp3} in terms of $\tu{hom}(\uH, \G_{n})$, the number of homomorphisms of $\uH$ in a graph on $n$ vertices (see \cite[Section~1.2]{BGLZ}). For convenience, we introduce the notation $\E_{\mu_{\bm \xi}}[N(\uH, \G_n)]\coloneqq \E_{\mu_{\bm \xi}}[N(\uH, \cdot)] $ where $\G_n$ denotes the simple unweighted random graph with edge distribution $\mu_{\bm \xi}$. 
		
	We next extend the notion of subgraph counts to weighted graphs. Let $\tu{G}_{n,{\bm \xi}}$ be the weighted graph on  vertex set $\set{n}$ with edge weights $(\xi_{i,j})_{1\leq i<j\leq n}={\bm \xi}$.
	We define $N(\uH, \tu{G}_{n,{\bm \xi}})$ to be the following sum:
	\[N(\uH, \tu{G}_{n,{\bm \xi}}) = \sum_{\stackrel{\phi:V(\uH)\rightarrow [n]}{\text{injective}}}\prod_{(u,v)\in E(\uH)}\xi_{\phi(u), \phi(v)}.\]

	We record two useful observations. First, 
	$\E_{\mu_{\bm \xi}}[N(\uH,\G_n)] = N(\uH, \tu{G}_{n, {\bm \xi}})\leq \hom(\uH, \tu{G}_{n, {\bm \xi}})$.
	Second, the optimizer ${\bm \xi}$ which asymptotically minimizes the variational problem \cite[Theorem~1.5]{BGLZ} also satisfies the inequality $\E_{\mu_{\bm \xi}}[N(\uH, \G_n)]\geq (1+\delta)n^{v_{\uH}}p^{e_\uH}$, under the same setting of Theorem \ref{thm:loc-1}.
	Consequently, we have 
	\begin{equation}\label{eq:Psi-limit}
	\lim_{n \to \infty} \frac{\Psi_{p,\uH}(\delta)}{n^{2}p^{\Delta}\log(1/p)}= 
	\theta_{\uH^{*}}. 
	\end{equation}
	Hence the mean-field approximation is asymptotically tight for the upper tail of any irregular subgraph count in the regime $n^{-1/\Delta} \ll p \ll 1$.
	
	It is natural to seek whether such a mean-field approximation also holds under the setting of Theorem \ref{thm:loc-2}. The following result shows that this is indeed the case for $H=K_{1,r}$. To ease notation, we write $\widehat \Psi_{p,r}(\delta)\coloneqq \Psi_{p, K_{1,r}}(\delta)$.



%

%

	\begin{theorem}[Mean-field approximation]\label{thm:mean-field}
		Let $r\geq 2$. Fix $\delta>0$ and $\varepsilon\in(0,1)$. Assume $p \ll 1$ and $n^{r+1}p^{r}\gg (\log n)^{r/(r-1)}$. Then for large $n$,
		\begin{equation*}
			(1-\varepsilon)\widehat\Psi_{p, r}\left(\delta(1-\varepsilon)\right)\leq -\log\P\left(N(K_{1,r}, \G(n,p))\geq(1+\delta)n^{r+1}p^{r}\right) \leq (1+\varepsilon)\widehat\Psi_{p,r}\left(\delta(1+\varepsilon)\right)
		\end{equation*}
	\end{theorem}

	In light of Theorem \ref{thm:mean-field} we make the following plausible conjecture.
	\begin{conjecture}
		For any strictly balanced graph $\uH$, $\varepsilon \in (0,1)$, $\delta >0$, and $p$ such that $p \ll n^{-1/\Delta}$ and $n^{v_{\uH}} p^{e_{\uH}} \gg (\log n)^{\alpha_\uH^*/(\alpha_\uH^*-1)}$, and for all large $n$
		\[
		(1-\varepsilon)\Psi_{p, \uH}\left(\delta(1-\varepsilon)\right)\leq -\log\P\left(N(\uH, \G(n,p))\geq(1+\delta)n^{v_{\uH}}p^{e_{\uH}}\right) \leq (1+\varepsilon)\Psi_{p,\uH}\left(\delta(1+\varepsilon)\right).
		\]
	\end{conjecture}

	\subsection*{Emergence of non-planted optimizers} 
	For $n^{-1/\Delta}\ll p\ll 1$, it was shown  in \cite{BGLZ} that $\Psi_{p,\uH}$ is optimized by some `planted' hub of appropriate size, for any connected irregular graph $\uH$. Namely, the optimizer ${\bm \xi}^* =\{\xi_{i,j}^*\}_{i < j \in \set{n}}$ is of the form $\xi_{i,j}^* \in \{p,1\}$ for all $i < j \in \set{n}$. In this case the term `planted' reflects that one can deterministically fix the edges of the hub (setting their weights to $1$), and the probability of observing this specific configuration already captures the correct large deviation speed and the rate function.
	
In contrast, in the regime $p \ll n^{-1/r}$ and $n^{r+1}p^r \gg (\log n)^{r/(r-1)}$ the optimizer $\widehat {\bm \xi}$ of $\widehat \Psi_{p,r}$ is a non-planted one, meaning that we no longer have $\widehat \xi_{i,j} \in \{p,1\}$ (see \eqref{eq:hat-xi}). This reflects a fundamentally different mechanism: the upper tail large deviations of $N(K_{1,r},\G(n,p))$ is driven by the presence of a vertex with high degree (as evidenced through Theorem \ref{cond_struc:star loc-2}), rather than by a fixed deterministic structure. While this event constrains the number of edges in the one-neighbourhood of a vertex, it does not specify which edges are present. Consequently, the neighbourhood remains random, and the correct rate function is obtained by enforcing a large-degree event rather than fixing specific edges. Thus the optimizer turns out to be `non-planted'.

%
%

\subsection{Complete irregular bipartite graphs} 
Although we have a complete understanding of the upper tail large deviations of subgraph counts for  any irregular connected $\tu{H}$ with maximum degree $\Delta$, in the regime $p \gg n^{-1/\Delta}$, beyond the case of star graphs, at the time of writing this paper, nothing is known in the range of $p$ such that $(\log n)^{\alpha^*_{\tu H}/(\alpha_{\tu H}^*-1)} \ll n^{v_{\tu H}}p^{e_{\tu H}}$ and $p \ll n^{-1/\Delta}$. We believe that the arguments employed in the proof of Theorems \ref{thm:loc-2} and \ref{cond_struc:star loc-2} can be extended to analyse the upper tail large deviations of the number of copies of the
irregular complete bipartite graphs $K_{s,t}$ with $s<t$ in the sparse regime $p \ll n^{-1/t}$. 
The upper tail behavior is expected to resemble that of $K_{1,r}$, with the dominant contribution arising from non-planted structures. Heuristically, the excess copies should be generated by a set of  $s$ vertices sharing an unusually large common neighborhood, of size at least $\frac{1}{s!}\,\delta^{1/t}\, n^{\frac{s}{t}+1} p^s$.
The large deviation speed predicted by this mechanism is consistent with the bounds in \cite{JOR}.
The techniques developed for $K_{1,r}$ in this paper, should extend to yield sharp asymptotics throughout the regime $n^{-1/t} \gg p \gg n^{-1/s - 1/t} (\log n)^{O(1)}$.
However, reaching the optimal regime --namely, obtaining the correct exponent in the logarithmic factor-- would require an appropriate extension of Lemma \ref{glow properties}. 

	\subsection{Notational conventions}\label{sec:notation}
	For any two sequences of positive real numbers $\{a_{n}\}$ and $\{b_{n}\}$, we write $b_{n}=o(a_{n})$, $a_n\ll b_{n}$, and $b_{n}\gg a_{n}$ to denote $\lim_{n\rightarrow\infty } a_{n}/b_{n} = 0 $. We use $a_{n} \lesssim b_{n}$ and $b_n \gtrsim a_n$ to denote $\limsup_{n\rightarrow \infty} a_{n}/b_{n}<\infty$.  
	We write $a_{n}\asymp b_{n}$ or $a_{n}=\Theta(b_{n})$ if $a_{n}\lesssim b_{n}$ and $a_n \gtrsim b_{n}$. For $x \in \R$ we use the standard notation $\lfloor x \rfloor$, $\lceil x \rceil$, and $\{x\}$ to denote the largest integer smaller than or equal to $x$, the smallest integer greater than or equal to $x$, and the fractional part of $x$, respectively.

	For a graph $\tu{G}$ we write $V(\tu{G})$ and $E(\tu{G})$ to denote its vertex set and edge set respectively. We will use $v(\tu{G}) = |V(\tu{G})|$ and $e(\tu{G}) = |E(\tu{G})|$, where $|\cdot|$ is the cardinality of a set. For typographic reasons we also use $v_{\tu{G}}$ and $e_{\tu{G}}$ instead of $v(\tu{G})$ and $e(\tu{G})$. 

	%

	\subsection*{Acknowledgements} Preliminary versions of Theorems \ref{thm:loc-1}, \ref{thm:poi}, and \ref{thm:loc-2} can be found on SK's
	\href{https://sites.google.com/icts.res.in/shaibalkarmakar/home}{webpage}. It has come to our recent notice that some of the results in this paper have been independently obtained in \cite{AS25} and in the recently posted work \cite{CAHMS}. We thank Wojciech Samotij for bringing the latter work to our attention before it was publicly posted. We also thank Nicholas Cook and Lutz Warnke for helpful comments{, and {two} anonymous referees for their valuable suggestions {that led to the improvements of presentation of this paper}.} Research was partially supported by DAE Project no.~RTI4001 via ICTS and ANRF ARG MATRICS grant (file no.~ANRF/ARGM/2025/001927/MTR). Research of AB was also partially supported by the Infosys Foundation via the Infosys-Chandrashekharan Virtual Centre for Random Geometry.

	%

	\section{Proofs of Theorems \ref{thm:loc-1} and \ref{cond_struc:loc-1}}
	We begin with the proof of Theorem \ref{thm:loc-1}. It requires a few notation. For any graph $\tu{G}\subseteq K_{n}$, we write 
	\begin{equation*}
		\E_{\tu{G}}\left[N(\uH, \G(n,p))\right]\coloneqq \E\left[N(\uH, \G(n,p))\ |\ \tu{G}\subseteq \G(n,p)\right].
	\end{equation*}
	The following optimization problem will be useful to prove Theorem \ref{thm:loc-1}:
	\begin{equation*}
		\Phi_{\uH}(\delta) \coloneqq \Phi_{p,\uH}(\delta) \coloneqq \inf\left\{e(\textup{G})\log(1/p):\ \textup{G}\subseteq K_n,\ \E_{\textup{G}}[N(\uH, \G(n,p))] \geq (1+\delta)\E[N(\uH, \G(n,p))]\right\}, \delta \geq 0.
	\end{equation*}
	Let $G\subseteq K_n$  be any graph in the feasible set of the above optimization problem. Define ${\bm \xi} = (\xi_{i,j})_{i<j\in \set{n}}$ by setting $\xi_{i,j} = 1$ if $(i,j)\in E(\tu{G})$ and $\xi_{i,j}=p$ otherwise. It is straightforward to verify that $\E_{\mu_{\bm \xi}}[N(\uH, \G_n)]\geq (1+\delta)\E_{\mu_p}[N(\uH,\G_n)]$ and $I_p({\bm \xi}) = e(\tu{G})\log(1/p)$. Consequently, one has $\Psi_{p,\uH}(\delta)\leq \Phi_{\uH}(\delta)$. 
	
	In addition, as the optimizer of $\Psi_{p,\uH}$ takes values in $\{p,1\}$ (recall the discussion above \eqref{eq:Psi-limit}) 
	which together with \eqref{eq:Psi-limit} allows us to conclude  
%
	\begin{equation}
		\label{limit:variational problem}
		\lim_{n\rightarrow \infty}\frac{\Phi_{\uH}(\delta)}{n^{2}p^{\Delta}\log(1/p)} = \theta_{\uH^{*}},
	\end{equation}
	for fixed $\delta>0$ and $n^{-1/\Delta} \ll p \ll 1$.
	We will also use a technical result from \cite{HMS}, which we restate in our graph setting:
	\begin{theorem}[{\cite[Theorem~3.1]{HMS}}]
		\label{HMS:technical result}
		For every graph $\uH$ and all positive real number $\varepsilon$ and $\delta$ with $\varepsilon<1/2$, there is a positive constant $K=K(e_{\uH}, \delta, \varepsilon)$ such that the following holds. Let $p\in (0,1-\varepsilon]$ and let $\Phi_{\uH}(\delta-\varepsilon)\geq K\log(1/p)$.
		Denote by $\I^{*}$ the collection of all graphs $\textup{G}\subseteq K_n$ 
		satisfying 
		\begin{itemize}
			\item[(C1)] $\E_{\textup{G}}\left[N(\uH, \G(n,p))\right]\geq (1+\delta-\varepsilon)\E\left[N(\uH,\G(n,p) )\right]$,
			\item[(C2)] $e(\textup{G})\leq K\cdot \Phi_{\uH}(\delta+\varepsilon)$, 
			
		\end{itemize}
		and 
		\begin{itemize}
			\item[(C3)]$\min_{e\in E(\textup{G})}\left(\mathbb{E}_{\textup{G}}\left[N(\uH, \G(n,p))\right]- \mathbb{E}_{\textup{G}\setminus e}\left[N(\uH,\G(n,p))\right]\right) \geq\mathbb{E}\left[N(\uH, \G(n,p))\right]\big/(K\cdot \Phi_{\uH}(\delta +\varepsilon))$.
		\end{itemize}
		Assume that $\I^{*}$ satisfies the following bound: for every integer $m$, 
		\begin{equation}
			\label{eq:entropic stability HMS}
			|\{\textup{G}\in \I^{*}: e(\textup{G})=m\}|\leq (1/p)^{\varepsilon m/2}.
		\end{equation}
		Then
		\begin{equation}
			\label{eq:tight with variational problem}
			(1-\varepsilon)\Phi_{\uH}(\delta-\varepsilon)\leq -\log\P\left(\UTH\right)\leq (1+\varepsilon)\Phi_{\uH}(\delta+\varepsilon)
		\end{equation}
		Further setting
		\[
		\mathcal{J}^{*}\coloneqq \left\{ \tu{G} \in \mathcal{I}^*: e(\tu{G})\log(1/p)\leq (1+\varepsilon)\Phi_{\uH}(\delta+\varepsilon) \right\},
		\]
		we have		
			\begin{equation}
				\label{cond_struc:step 1}
				\P\left(\UTH \cap \left\{\tu{G}\nsubseteq \G(n,p)\ \forall \, \tu{G}\in \mathcal{J}^{*}\right\}\right) \leq (\P(\UTH))^{1+\varepsilon/16}.
			\end{equation}
		\end{theorem}
		
%

\begin{remark}
		\label{rmk:HMS-improve}
		\cite[Theorem~3.1]{HMS} yields \eqref{cond_struc:step 1} with its RHS replaced by $\varepsilon \P(\UTH)$. We  need \eqref{cond_struc:step 1} in its current form to prove Theorems \ref{cond_struc:loc-1} and \ref{thm: cond struc regular graph}(a). It is a straightforward adaptation of \cite[Theorem~3.1]{HMS}. Here is a brief outline: By \cite[Eqn.~(16)]{HMS}
		\begin{equation}
			\label{eq: HMS eq 16}
			\P(\UTH)\geq \exp\left(-(1+\varepsilon/8)\cdot\Phi_{\uH}(\delta+\varepsilon)\right).
		\end{equation}
On the other hand, setting $\ell:= \lceil \varepsilon (3e_{\tu H})^{-1} K \Phi_{\tu H}(\delta+\varepsilon)\rceil$, and letting
\[
\mathcal{I}:= \{\tu G \subseteq K_n: e(G) \leq e_{\tu H} \ell \text{ and (C1) holds}\},
\]
we apply \cite[Lemma 3.7]{HMS} to deduce that
\[
\P\left(\UTH \cap \left\{\tu{G}\nsubseteq \G(n,p)\ \forall \, \tu{G}\in \mathcal{I}\right\}\right) \leq \left(1-\frac{\varepsilon/2}{1+\delta}\right)^\ell \leq \left(\P(\UTH\right)^2,
\]		
where we use \eqref{eq: HMS eq 16} and choose $K$ sufficiently large depending on $\varepsilon$ and $e_{\tu H}$. The last inequality together with \cite[Lemma 3.8]{HMS} further yields that
\begin{equation}\label{eq:cI}
\P\left(\UTH \cap \left\{\tu{G}\nsubseteq \G(n,p)\ \forall \, \tu{G}\in \mathcal{I}^*\right\}\right) \leq \left(\P(\UTH\right)^2.
\end{equation}
Compare \cite[Eqn.~(17)]{HMS} with \eqref{eq:cI}. Now, using \eqref{eq:cI} (instead of \cite[Eqn.~(17)]{HMS}) and repeating the rest of the arguments in the proof of \cite[Theorem 3.1]{HMS} the claimed bound in \eqref{cond_struc:step 1} follows. 
\end{remark}

		Note that the statements \eqref{limit:variational problem} and \eqref{eq:tight with variational problem} together imply Theorem \ref{thm:loc-1}. Also by \eqref{limit:variational problem}, we have $\Phi_{\uH}(\delta)\geq K\cdot \log(1/p)$, where $K$ is the constant of Theorem \ref{HMS:technical result}. We only need to show \eqref{eq:entropic stability HMS} 
		holds in the regime $n^{-1/\Delta}\ll p\ll 1$. Before we proceed we will need a couple of definitions and combinatorial results. 
		\begin{definition}\label{def:Q_H}
			Let $\uH$ be a connected graph with maximum degree $\Delta$. Define 
			\begin{equation*}
				\begin{aligned}
					Q_{\uH} \coloneqq \{\tu{J}: &\textup{ $\tu{J}$ is a non-empty subgraph of $\uH$ without isolated vertices and } \\ &\textup{admits a bipartition $V(\tu{J})=A\cup B$ such that $\tu{deg}_{\tu{J}}(a) = \Delta$ for all $a\in A$} \}.
				\end{aligned}
			\end{equation*}
		\end{definition}
		In the anticipation of the use of Definition \ref{def:Q_H} in the proof of Theorem \ref{thm: cond struc regular graph} we have allowed $\tu{H}$ to be regular there. In the rest of this section, unless mentioned otherwise, we will implicitly assume that $\tu{H}$ is irregular.

		\begin{lemma}
			\label{s1:simplelowerbound}
			Let $H$ be a connected, irregular graph with maximum degree $\Delta$. 
			\begin{enumerate}
				\item[(i)] Let $\tu{J}$ be a non-empty subgraph of $H$ without isolated vertices. Then $			\Delta(v_{\tu{J}}-\alpha^{*}_{\tu{J}})-e_{\tu{J}}\geq 0$,
				and the equality holds if and only if  $\tu{J}\in Q_{\uH}$.
				\item[(ii)] Fix any $\tu{J} \in Q_{\uH}$. Let $(a,b) \in E(\tu{J})$ for some $a \in A$ and $b \in B$. Define $\widehat{\tu{J}}$ to be the subgraph obtained from $\tu{J}$ by deleting all edges incident to $\{a,b\}$. Then $\alpha^*_{\widehat{\tu{J}}}= |B|-1$.
			\end{enumerate}
		\end{lemma}
		The proof of Lemma \ref{s1:simplelowerbound}(i) uses an argument similar to that of \cite[Lemma~5.3]{HMS}. The proof of Lemma \ref{s1:simplelowerbound}(ii) uses the fact that the fractional independence number of a bipartite graph must be an integer.
		\begin{proof}
			Assume that $\tu{J}$ is a non-empty subgraph of $\uH$ without isolated vertices. By \cite[Lemma~5.2]{HMS} $\tu{J}$ has a fractional independent set $\alpha:V(\tu{J})\rightarrow [0,1]$ such that $\alpha(v)\in \left\{0,\frac{1}{2}, 1\right\}$ and $\sum_{v\in V(\tu{J})}\alpha(v) = \alpha^{*}_{\tu{J}}$. Then,
			\begin{equation}
				\label{local_ref}
				e_{\tu{J}} \leq \sum_{(u,v) \in E(\tu{J})}(2-\alpha(u)-\alpha(v)) =\sum_{u\in V(\tu{J})}(1-\alpha(u))\text{deg}_{\tu{J}}(u)\leq \Delta\cdot \sum_{u\in V(\tu{J})}(1-\alpha(u)) = \Delta\cdot(v_{\tu{J}}-\alpha^{*}_{\tu{J}}).
			\end{equation}
			Suppose $\Delta(v_{\tu{J}}-\alpha^{*}_{\tu{J}})= e_{\tu{J}}$, then both inequalities in \eqref{local_ref} are equalities. We must have $\alpha(u)+\alpha(v) = 1$ for every $(u,v) \in E(\tu{J})$ and $\text{deg}_{\tu{J}}(u) = \Delta$ whenever $\alpha(u)\neq 1$. Let $A = \alpha^{-1}(\{0\})$, $B = \alpha^{-1}(\{1\})$ and $C = \alpha^{-1}(\{1/2\})$. Every vertex in $A\cup C$ has degree $\Delta$. Each edge of $\tu{J}$ has either both endpoints in $C$ or one endpoint in $A$ and other endpoint in $B$. If $C \neq \phi$, then it induces a $\Delta$-regular subgraph of $\uH$ and $\uH$ being connected with maximum degree $\Delta$ implies that $V(\uH)=C$, which contradicts that $\uH$ is irregular. Thus $C=\phi$ and hence 
			$\tu{J}\in Q_{\uH}$.
			
			Conversely, if $\tu{J} \in Q_{\uH}$ set 
			$\alpha(a) = 0$ for $a\in A$ and $\alpha(b) = 1$ for $b\in B$. Such an $\alpha$ is fractional independent set in $\tu{J}$ and we have $e_{\tu{J}}= \Delta (v_{\tu{J}} - \bar\alpha)$, where $\bar \alpha \coloneqq \sum_{v} \alpha(v)$. This establishes part (i).
			
			Turning to prove the second part we notice that \eqref{local_ref} continues to hold even if $J$ contains isolated vertices. This, in particular implies that $e_{\widehat{\tu{J}}} \le \Delta (v_{\widehat{\tu{J}}} - \alpha^*_{\widehat{\tu{J}}})$. On the other hand, $e_{\widehat{\tu{J}}} = e_{\tu{J}} - \deg_{\tu{J}}(a) - \deg_{\tu{J}}(b) +1$ and $e_J = \Delta |A|$. Combining these observations, as $a \in A$ and $\deg_{\tu{J}}(b) \le \Delta$, we find that 
			\begin{equation*}
				\alpha^*_{\widehat{\tu{J}}} \le v_{\widehat{\tu{J}}} - e_{\widehat{\tu{J}}}/\Delta = |A|+|B| -2 - \frac{\Delta(|A|-1) - (\deg_{\tu{J}}(b) -1)}{\Delta} < |B|.
			\end{equation*}
			Since $\widehat{\tu{J}}$ is a bipartite graph it follows from \cite[Proposition A.3]{JOR} that $\alpha^*_{\widehat{\tu{J}}}$ must be an integer and thus $\alpha^*_{\widehat{\tu{J}}} \le |B|-1$. On the other hand, setting $\alpha(u)=0$ for $u \in A\setminus \{a\}$ and $\alpha(u)=1$ for $B \setminus \{b\}$ we see that $\alpha$ is a fractional independent set of $\widehat{J}$ for which $\bar \alpha = |B|-1$ proving that $\alpha^*_{\widehat{\tu{J}}}=|B|-1$.
		\end{proof}
	
%
%
		The next two lemmas together show that in any $\tu{G}\in \I^{*}$, most edges join vertices whose degrees sum to a quantity of order almost $n$. For any $e\in E(\tu{G})$, we use the notation $N(\uH, \tu{G},e)$ to denote the number of labelled copies of $\uH$ in $\tu{G}$ that uses the edge $e$. 
		\begin{lemma}
			\label{s1:excluding edges}
			Fix $\cbar>0$. 
			Let $\tu{G}\subseteq K_{n}$ be such that 
			\begin{itemize}
				\item[(i).] $e(\tu{G})\leq \cbar n^{2}p^{\Delta}\log(1/p)$
			\end{itemize}
			and
			\begin{itemize}
				\item[(ii).] For every $e\in E(\tu{G})$,  
				\[
				\sum_{\phi\neq J\subseteq H}N(\tu{J}, \tu{G}, e)\cdot n^{-v_{\tu{J}}}p^{-e_{\tu{J}}}\geq 1/ \left(2 \cbar n^{2}p^{\Delta}\log(1/p)\right),
				\]
				where the sum is taken over all non-empty subgraph $\tu{J}$ of $\uH$ without any isolated vertices.
			\end{itemize}
			Define $E_{\text{exc}}(\tu{G})$ to be the subset of edges of $\tu{G}$ such that for $e\in E_{\text{exc}}(\tu{G})$
			\begin{equation*}
				\sum_{\tu{J}\in Q_{\uH}}N(\tu{J}, \tu{G}, e)\cdot n^{-v_{\tu{J}}}p^{-e_{\tu{J}}} < 1/\left(4\cbar n^{2}p^{\Delta}\log(1/p)\right).
			\end{equation*}
			Then $|E_{\text{exc}}(\tu{G})|\leq \widehat{C}\left(\log(1/p)\right)^{v_{\uH}+1} \cdot n^{2}p^{\Delta}\cdot p^{\sigma} \eqqcolon m_0$ for some constants $\widehat{C}$ (depending only on $\uH$ and $\cbar$) and $\sigma$ (depending only on $\uH$).
		\end{lemma}
		As we will see below, $E_{\text{exc}}(\tu{G})$ represents a collection of ``exceptional'' edges whose removal will allow us to control the sum of the degrees of the endpoints of every remaining edge. Lemma \ref{s1:excluding edges} shows that the collection of these edges is negligible as soon as $p \ll 1$.
		
		\begin{proof}
			Define 
			\[
			\sigma \coloneqq \min\{\Delta(v_{\tu{J}}-\alpha^{*}_{\tu{J}})-e_{\tu{J}}:\ \phi \neq \tu{J}\subseteq \uH \text{ without isolated vertices and $\tu{J}\notin Q_{\uH}$}\}. 
			\]
			By Lemma \ref{s1:simplelowerbound}, $\sigma>0$. By definition of $E_{\text{exc}}(\tu{G})$ we have
			\begin{equation}
				\label{s1:firsteq}
				\frac{|E_{\text{exc}}(\tu{G})|}{4 \cbar n^{2}p^{\Delta}\log(1/p) }\leq\sum_{e\in E_{\text{exc}}(\tu{G})}\sum_{\substack{\phi\neq \tu{J}\subseteq H\\ \tu{J}\notin Q_{\uH}}} \frac{N(\tu{J}, \tu{G}, e)}{n^{v_{\tu{J}}}p^{e_{\tu{J}}}} \leq 2e_{\tu{J}}\sum_{\substack{\phi\neq \tu{J}\subseteq H\\ \tu{J}\notin Q_{\uH}}} \frac{N(\tu{J},\tu{G})}{{n^{v_{\tu{J}}}p^{e_{\tu{J}}}}}.
			\end{equation}
			Using the upper bound on $e(\tu{G})$ 
			along with \cite[Theorem~5.4]{HMS} and the definition of $\sigma$, we therefore get 
			\begin{equation}
				\label{s1:2nd eq}
				N(\tu{J}, \tu{G})\leq (2\cbar n^{2}p^{\Delta}\log(1/p))^{v_{\tu{J}}-\alpha^{*}_{\tu{J}}}\cdot  n^{2\alpha^{*}_{\tu{J}}-v_{\tu{J}}} \leq (2\cbar \log(1/p))^{v_{\uH}}\cdot  {n^{v_{\tu{J}}}p^{e_{\tu{J}}}} \cdot p^{\sigma},
			\end{equation}
			for any non-empty $\tu{J} \subseteq \uH$ without any isolated vertices 
			such that $\tu{J}\notin Q_{\uH}$. Finally, combining \eqref{s1:firsteq} and \eqref{s1:2nd eq} completes the proof.
		\end{proof}
		\begin{lemma}
			\label{s1:sum of degrees}
			Assume the same setup as in Lemma \ref{s1:excluding edges}. Further assume $p\ll1 $. Then for every edge $e=(u,v)\in E(\tu{G})\setminus E_{\text{exc}}(\tu{G})$, we have
			\begin{equation*}
				\textup{deg}_{\tu{G}}(u)+\textup{deg}_{\tu{G}}(v) \geq \frac{c_{0}}{\left(\log(1/p)\right)^{v_{\uH}}}\cdot n,
			\end{equation*}
			for some positive constant $c_{0}$ (depending only on $\uH$ and $\cbar$).
		\end{lemma}
		\begin{proof}Let $e=(u,v)\in E(\tu{G})\setminus E_{\text{exc}}(\tu{G})$.
			By the definition of $E_{\text{exc}}(\tu{G})$  there exists $ \tu{J}\in Q_{\uH}$ such that
			\begin{equation}
				\label{s1:dp eq1}
				N(\tu{J}, \tu{G}, uv)\cdot n^{-v_{\tu{J}}}p^{-e_{\tu{J}}}\geq 
				\frac{1}{4 |Q_{\uH}| \cbar n^{2}p^{\Delta}\log(1/p)}.
			\end{equation}
			By the definition of $Q_{\uH}$, $\tu{J}$ admits a bipartition $V(\tu{J})=A\cup B$ where $\text{deg}_{\tu{J}}(a) = \Delta$ for $a\in A$.
			Now we use \cite[Lemma~5.8]{HMS}, which gives us
			\begin{equation}
				\label{s1:dp eq2}
				N(\tu{J}, \tu{G},uv)\leq e_{\tu{J}}\cdot \left(\text{deg}_{\tu{G}}(u)+\text{deg}_{\tu{G}}(v)\right)\cdot \left(2e(\tu{G})\right)^{|A|-1}\cdot n^{|B|-|A|-1}.
			\end{equation}
			Using $e(\tu{G})\leq \cbar n^{2}p^{\Delta}\log(1/p)$ in \eqref{s1:dp eq2} and applying it to \eqref{s1:dp eq1}, we get
			\begin{equation*}
				\text{deg}_{\tu{G}}(u)+\text{deg}_{\tu{G}}(v) \geq \frac{1}{2e_{\uH}|Q_{\uH}|(2\cbar)^{v_{\uH}}}\frac{n}{\left(\log(1/p)\right)^{v_{\uH}}}, 
			\end{equation*}
			where we have used that $e_{\tu{J}} = \Delta |A|$, $v_{\tu{J}} = |A|+|B|$ and $|A|\leq v_{\uH}$.
		\end{proof}
		Now, we are ready to prove Theorem \ref{thm:loc-1}. 
		\begin{proof}[Proof of Theorem \ref{thm:loc-1}]
			Recall that we only need to show \eqref{eq:entropic stability HMS} holds in the regime $n^{-1/\Delta}\ll p\ll 1$. 
			By 
			\eqref{limit:variational problem}, we have $\Phi_{\uH}(\delta)  \asymp n^{2}p^{\Delta}\log(1/p)$.  
			Using this and the conditions (C1) and (C2), 
			there exists $ \cbar\coloneqq \cbar(\uH, \delta, \varepsilon)>0$ such that any $\tu{G}\in \I^{*}$ satisfies
			\begin{equation*}
				m_{\min} \coloneqq\ \frac{1}{\cbar} n^{2}p^{\Delta}\ \leq\ e(\tu{G})\ \leq K\cdot \Phi_{\uH}(\delta+\varepsilon)\ \leq\ \cbar n^{2}p^{\Delta}\log(1/p) \eqqcolon m_{\max}.
			\end{equation*}
			Let $m$ be any integer satisfying $m_{min}\leq m\leq m_{\max}$. Define 
			$ \I^{*}_{m} \coloneqq \{\tu{G}\in \I^{*}:\ e(\tu{G}) = m\}$.
			Let $\tu{G}\in \I^{*}_{m}$. Condition (C3) gives us that for every $e\in E(\tu{G})$,
			\begin{equation*}
				\frac{\E[N(\uH,\G(n,p))]}{\cbar n^{2}p^{\Delta}\log(1/p)}\leq \E_{\tu{G}}[N(\uH, \G(n,p))] - \E_{\tu{G}\setminus e}[N(\uH, \G(n,p))] \leq \sum_{\phi\neq \tu{J}\subseteq H} N(\tu{J}, \tu{G},e)\cdot n^{v_{\uH}-v_{\tu{J}}}p^{e_{\uH}-e_{\tu{J}}},
			\end{equation*}
			where the sum ranges over the nonempty subgraphs $\tu{J}$ of $\tu{H}$ without isolated vertices. Since $\E[N(\uH,\G(n,p))] = (1+o(1))n^{v_{\uH}}p^{e_{\uH}}$, for sufficiently large $n$ we have
			\begin{equation*}
				\sum_{\phi\neq \tu{J}\subseteq H} \frac{N(\tu{J},\tu{G},e)}{n^{v_{\tu{J}}}p^{e_{\tu{J}}}} \geq \frac{1}{2 \cbar n^2p^{\Delta}\log(1/p)}, \qquad \text{for every $e\in E(\tu{G})$}.
			\end{equation*}
			Therefore, 
			both Lemma \ref{s1:excluding edges} and Lemma \ref{s1:sum of degrees} apply to any $\tu{G} \in \I^*_m$. This, in particular, implies that $|E_{\text{exc}}(\tu{G})|\leq m_0$ and 
				for every $e=(u,v)\in E(\tu{G})\setminus E_{\text{exc}}(\tu{G})$,
				\begin{equation*}
					\text{deg}_{\tu{G}}(u) + \text{deg}_{\tu{G}}(v) \geq \frac{c_{0}}{\left(\log(1/p)\right)^{v_{H}/2}}\cdot n.
				\end{equation*}
				Define 
				\begin{equation*}
					B_{*}(\tu{G}) \coloneqq \{v\in V(\tu{G}): \text{deg}_{\tu{G}}(v)\geq np^{\varepsilon}\}.
				\end{equation*}
				We claim that each edge $e=(u,v)\in E(\tu{G})\setminus E_{\text{exc}}(\tu{G})$ has at least one endpoint in $B_{*}(\tu{G})$. Otherwise, as $p\ll 1$ 
				\begin{equation*}
					\text{deg}_{\tu{G}}(u)+\text{deg}_{\tu{G}}(v)\leq 2np^{\varepsilon}\ll \frac{c_0}{\left(\log(1/p)\right)^{v_{H}/2}}\cdot n,
				\end{equation*}
				which is a contradiction. Another easy bound that we get is 
				\begin{equation*}
					2m \geq \sum_{v\in B_{*}(\tu{G})}\text{deg}_{\tu{G}}(v) \geq |B_{*}(\tu{G})|\cdot np^{\varepsilon}
				\end{equation*}
				and hence $|B_{*}(\tu{G})|\leq 2mn^{-1}p^{-\varepsilon}$. We now derive an upper bound on $|\I^{*}_{m}|$ by counting the number of ways to construct a graph $\tu{G}\in \I^{*}_{m}$:
				\begin{itemize}
					\item[(1)] Choose some $m_{\text{exc}}\leq m_0$ 
					and then choose $m_{\text{exc}}$ edges of $K_{n}$ to form $E_{\text{exc}}(\tu{G})$.
					\item[(2)]Choose the set $B_{*}(\tu{G})$ and then choose rest of $m-m_{\text{exc}}$ edges from the set 
					\begin{equation*}
						\mathfrak{B} \coloneqq \{(u,v)\in E(K_{n}):u\in B_{*}(\tu{G})\}.
					\end{equation*}
				\end{itemize}
				The number of ways to choose $B_{*}(\tu{G})$ is at most 
				\begin{equation*}
					{n\choose 2mn^{-1}p^{-\varepsilon}}\leq n^{2mn^{-1}p^{-\varepsilon}} \leq p^{-\varepsilon m},
				\end{equation*}
				where we have used $p\geq n^{-1/\Delta}$. We further have $|\mathfrak{B}|\leq |B_{*}(\tu{G})|\cdot n \leq 2mp^{-\varepsilon}$. Therefore, using $p \ll 1$ yet again, 
				\begin{equation}
					\label{s1:main ineq}
					|\I^{*}_{m}|\leq p^{-\varepsilon m}\cdot \sum_{m_{\text{exc}}=0}^{m_{0}}{n^{2}\choose m_{\text{exc}}}\cdot { 2mp^{-\varepsilon}\choose m-m_{\text{exc}}}\leq p^{-3\varepsilon m} \sum_{m_{\text{exc}}=0}^{m_{0}}{n^{2}\choose m_{\text{exc}}}.
				\end{equation}
				Now using the inequality $\sum_{i=0}^{k}{m\choose i}\leq(me/k)^{k}$ and the fact that $m_0 \ll m_{\min} \le m$ we  finally get
				\begin{equation}
					\label{s1:imp 3 eq}
					\sum_{m_{\text{exc}}=0}^{m_{0}}{n^{2}\choose m_{\text{exc}}}\leq \left(\frac{en^{2}}{m_{0}}\right)^{m_{0}}\leq p^{-2\varepsilon m}.
				\end{equation}
				Plugging \eqref{s1:imp 3 eq} in \eqref{s1:main ineq} gives $|\I^{*}_{m}|\leq p^{-5\varepsilon m}$. This completes the proof.
			\end{proof}
			
			We now turn to the prove Theorem \ref{cond_struc:loc-1}. To this end, for every $\varepsilon>0$ define
			\begin{equation*}
				\begin{aligned}
					\tu{Near-min}(\varepsilon) \coloneqq &\{\tu{G}\subseteq K_{n}:\ \tu{G} \in \mathcal{I}^* \text{ and } e(\tu{G})\leq (1+\mathfrak{t}_0 (\varepsilon))\theta_{\uH^{*}} n^{2}p^{\Delta} \},
				\end{aligned}
			\end{equation*}
			where $\mathfrak{t}_0(\varepsilon) = \mathfrak{t}_0(\uH, \delta, \varepsilon)$ is some constant, to be determined below, such that $\lim_{\varepsilon \downarrow 0} \mathfrak{t}_0(\varepsilon) =0$. By \eqref{limit:variational problem} and \eqref{cond_struc:step 1}, and the continuity of the map $\delta \mapsto \theta_{\uH^*}(\delta)$, in the regime $n^{-1/\Delta}\ll p\ll 1$ we get, 
			\begin{equation*}
				\P\left(\left\{\exists \tu{G}\subseteq  \G(n,p): \tu{G}\in\tu{Near-min}(\varepsilon)\right\}\big| \UTH\right)\geq 1-(\P(\UTH))^{\varepsilon/16},
			\end{equation*}
			for some appropriately chosen $\mathfrak{t}_0(\varepsilon)$. Hence, to complete the proof of Theorem \ref{cond_struc:loc-1} it remains to show that for any $\tu{G} \in \tu{Near-min}(\varepsilon)$ there must exist $W \subseteq V(\tu{G})$ such that
			\begin{equation}\label{eq:W-bd}
				\min_{w \in W} \deg_{\tu{G}}(w) \ge (1-\widetilde{\mathfrak{t}}(\varepsilon))n \qquad \text{ and } \qquad e(\tu{G}[W, V(\tu{G})\setminus W]) \geq (\theta_{\uH^*} - \mathfrak{t}(\varepsilon))n^2 p^\Delta,
			\end{equation}
			for some $\mathfrak{t}(\varepsilon)$ and $\widetilde{\mathfrak{t}}(\varepsilon)$ such that $\lim_{\varepsilon \downarrow 0} \max\{\mathfrak{t}(\varepsilon), \widetilde{\mathfrak{t}}(\varepsilon)\} =0$.
			
%
%
			
			The main tool is the following lemma. Before continuing, we recall from Definition \ref{def:Q_H} that every subgraph $\tu{J}\in Q_{\uH}$ admits a bipartition $V(J) = A\cup B$ such that $\tu{deg}_{\tu{J}}(a)=\Delta$ for all $a\in A$.
			\begin{lemma}
				\label{cond_struc: bipartite lemma}
				Assume $ p\ll 1$. Let $\varepsilon> 0$ and $\tu{G}\subseteq K_{n}$ be such that $e(\tu{G})\lesssim  n^{2}p^{\Delta}$.
				There exists a constant $\widetilde \eta\coloneqq \widetilde \eta(\varepsilon,\uH)>0$ such that for every $\eta \in (0,\widetilde \eta)$ with $U:=\{v\in V(\tu{G}): \tu{deg}_{\tu{G}}(v)\geq \eta n\}$ and $V \coloneqq V(\tu{G})\setminus U$, we have
				\begin{equation*}
					N_{U}(\tu{J}, \tu{G}[U,V]) \geq N(\tu{J},\tu{G})-\varepsilon n^{v_{\tu{J}}}p^{e_{\tu{J}}}, \qquad \forall\ \tu{J}\in Q_{\tu{H}},
				\end{equation*} 
				where 
				$N_{U}(\tu{J},\tu{G}[U,V])$ denotes the number of labelled copies of $\tu{J}$ in $\tu{G}[U,V]$ such that the vertices of $A$ are mapped to those in $U$. 
			\end{lemma}

			Deferring the proof of Lemma \ref{cond_struc: bipartite lemma} to later we now prove Theorem \ref{cond_struc:loc-1}.
			
			\begin{proof}[Proof of Theorem \ref{cond_struc:loc-1}]
				For any $\tu{G} \in \mathcal{I}^*$ we have
				\begin{equation}\label{eq:J-sum}
					(\delta -2\varepsilon) n^{v_{\uH}}p^{e_{\uH}} \leq \E_{\tu{G}}[N(\uH, \G(n,p))]-\E[N(H, \G(n,p))]\leq \sum_{\phi\neq \tu{J}\subseteq \uH} N(\tu{J}, \tu{G})\cdot  n^{v_{\uH}-v_{\tu{J}}} p^{e_{\uH}-e_{\tu{J}}},
				\end{equation}
				where the sum is over nonempty subgraph $\tu{J}$ of $\uH$ without any isolated vertices. By \eqref{s1:2nd eq} and 
				%
				Lemma \ref{cond_struc: bipartite lemma} 
				we further obtain
				\begin{equation}
					\label{cond_struc:step 3}
					\sum_{\tu{J}\in Q_{\uH}} N_{U}(\tu{J}, \tu{G}[U,V])\cdot n^{-v_{\tu{J}}} p^{-e_{\tu{J}}}\geq \delta -4\varepsilon.
				\end{equation}
				We claim that
				\begin{equation}
					\label{cond_struc:step 4}
					N_{U}(\tu{J}, \tu{G}[U,V])\leq\ \left(e(\tu{G}[U,V])\right)^{|A|} \cdot n^{|B|-|A|} = 	\left(e(\tu{G}[U,V])\right)^{v_{\tu{J}}-\alpha_{\tu{J}}^{*}} \cdot n^{2\alpha_{\tu{J}}^{*}-v_{\tu{J}}}, \qquad \forall \, J \in Q_{\uH},
				\end{equation}
				where we recall the definition of $Q_{\uH}$ from Definition \ref{def:Q_H}. To see \eqref{cond_struc:step 4} let $M$ be a matching of $\tu{J}$ of size $|A|$. Each edge in the matching has at most $e(\tu{G}[U, V])$ choices (restricted by the fact that $A$ must be mapped to $U$). There are $n^{|B|-|A|}$ many choices for the remaining $|B|-|A|$ vertices in $B$. The equality in \eqref{cond_struc:step 4} follows once we note that $\alpha_{\tu{J}}^{*} = |B|$ for all $J \in Q_{\uH}$. Now using  \eqref{cond_struc:step 4} in \eqref{cond_struc:step 3},
				\begin{equation}
					\label{cond_struc:step 5}
					\delta - 4\varepsilon \leq \sum_{\tu{J}\in Q_{\uH}} \frac{N_{U}(\tu{J}, \tu{G}[U,V])} {n^{v_{\tu{J}}}p^{e_{\tu{J}}}}\leq\sum_{\tu{J}\in Q_{\uH}} \left(\frac{e(\tu{G}[U,V])}{n^{2}p^{\Delta}} \right)^{v_{\tu{J}}-\alpha_{\tu{J}}^{*}} = P_{\uH^{*}}\left(\frac{e(\tu{G}[U,V])}{n^{2}p^{\Delta}} \right)-1,
				\end{equation}
				where the last step follows from the fact that there is an one-to-one correspondence between independent sets in $\uH^*$ of size $k$ and graphs $\tu{J} \in Q_\uH$ with $|A|= v_{\tu{J}} - \alpha^*_{\tu{J}}$.
				
				Further, using $e(\tu{G}[U,V])\leq e(\tu{G})$ and the monotonicity of $P_{\uH^{*}}$, as $\tu{G} \in \tu{Near-min}(\varepsilon)$, we find that 
				\begin{equation}
					\label{cond_struc:step 7}
					P_{\uH^{*}}\left(\frac{e(\tu{G}[U,V])}{n^{2}p^{\Delta}} \right)
					\leq  1+\delta+\mathfrak{t}_{1}(\varepsilon), 
				\end{equation}
				for some non-negative function $\mathfrak{t}_{1}(\cdot) = \mathfrak{t}_1(\cdot, \uH, \delta)$ such that $\lim_{\varepsilon\downarrow 0} \mathfrak{t}_{1}(\varepsilon)=0$. Since the second inequality in \eqref{cond_struc:step 5} holds term-wise we conclude from \eqref{cond_struc:step 5} and \eqref{cond_struc:step 7}, that 
				\begin{equation} 
					\label{cond_struc:step 8}
					N_{U}(K_{1, \Delta}, \tu{G}[U,V])\geq e(\tu{G}[U,V])\cdot n^{\Delta-1} - (\mathfrak{t}_{1}(\varepsilon)+4\varepsilon) n^{\Delta+1}p^{\Delta},
				\end{equation} 
				as $K_{1,\Delta}\in Q_{\uH}$. Take $\gamma \coloneqq \sqrt{\mathfrak{t}_{1}(\varepsilon)+4\varepsilon}$. Let $W\coloneqq \{v\in U: \tu{deg}_{\tu{G}}(v)\geq (1-\gamma)n\}\subseteq U$ and set $\beta \coloneqq e(\tu{G}[W,V])/ e(\tu{G}[U,V])$. Notice that
				\begin{equation}
					\label{cond_struc:step 9}
					\begin{aligned}
						&N_{U}(K_{1,\Delta}, \tu{G}[U,V]) = N_{W}(K_{1,\Delta}, \tu{G}[W,V]) + N_{U\setminus W}(K_{1,\Delta}, \tu{G}[U\setminus W,V]) \\
						& \leq e(\tu{G}[W,V]) \cdot n^{\Delta-1} + e(\tu{G}[U\setminus W,V])\cdot \left((1-\gamma)n\right)^{\Delta-1}\leq \left(1-\gamma(1-\beta)\right)\cdot e(\tu{G}[U,V]) \cdot n^{\Delta-1},
					\end{aligned}
				\end{equation}
				where the first inequality is obtained by observing that once we fix the image of any one edge of $K_{1,\Delta}$, the remaining $\Delta-1$ many vertices can have at most $n^{\Delta-1}$ or $((1-\gamma)n)^{\Delta-1}$ choices depending on where the centre vertex of $K_{1,\Delta}$ is mapped to $W$ or $U\setminus W$ respectively. 
				
				To conclude the proof we recall that $\theta_{\uH^{*}}$ is the unique positive solution to $P_{\uH^{*}}(\theta) = (1+\delta)$ and use that  $P_{\uH^{*}}$ is strictly increasing and continuous on $[0,\infty)$ to deduce from  \eqref{cond_struc:step 5} that 
				\begin{equation}
					\label{cond_struc:step 6}
					e\left(\tu{G}[U,V]\right)\geq (\theta_{\uH^{*}}- \mathfrak{t}_2(\varepsilon)) n^{2}p^{\Delta},
				\end{equation} 
				for some non-negative function $\mathfrak{t}_2(\cdot) = \mathfrak{t}_2(\cdot, \uH, \delta)$ such that $\lim_{\varepsilon \downarrow 0} \mathfrak{t}_2(\varepsilon)=0$. 
				Combining the lower and upper bounds in \eqref{cond_struc:step 8} with \eqref{cond_struc:step 9} and using \eqref{cond_struc:step 6} we conclude that 
				\begin{equation}
					e(\tu{G}[W,V])\geq \left(\theta_{\uH^{*}}-\mathfrak{t}_2(\varepsilon)-\sqrt{\mathfrak{t}_{1}(\varepsilon)+4\varepsilon}\right) n^{2}p^{\Delta}.
				\end{equation}
				This yields \eqref{eq:W-bd} completing the proof of the theorem.
			\end{proof}
			
			It remains to prove Lemma \ref{cond_struc: bipartite lemma}. It uses some results on $2$-matchings from \cite[Section 7]{BGLZ}. A $2$-matching of a graph is simply a union of two matchings of it. Notice that for any $J \in Q_{\uH}$ the set $A$ is a vertex cover of $\tu{J}$ (a subset of vertices that intersects every edge) with the property that $\Delta |A| = e_{\tu{J}}$. Therefore, by \cite[Lemma 7.1]{BGLZ}, for any $\tu{J} \in Q_{\uH}$ and any vertex $v$ of $\tu{J}$, there exists a $2$-matching $M$ consisting of $2e_{\tu{J}}/\Delta$ edges in which the connected component containing $v$ is a path. This observation will be used in the proof below.
			
			\begin{proof}[Proof of Lemma \ref{cond_struc: bipartite lemma}]
				Let $\tu{G}\subseteq K_{n}$ be such that $e(\tu{G})\leq C n^{2}p^{\Delta}$ for some constant $C< \infty$. Set $\eta>0$ such that
				\begin{equation}
					\label{eta:defn}
					\eta < \frac{\varepsilon}{2v_{\uH}\cdot (2C)^{2v_{\uH}}}.
				\end{equation}
				Define $U\coloneqq \{v\in V(\tu{G}): \tu{deg}_{\tu{G}}(v)\geq \eta n\}$ and $V\coloneqq V(\tu{G})\setminus U$. Note that $|U| \leq 2e(\tu{G})/(\eta n)$.
				Next for any $\tu{J}\in Q_{\uH}$ we 
				observe that
				\begin{equation}\label{eq:Psi-Phi}
					N(\tu{J}, \tu{G}) - N_{U}(\tu{J}, \tu{G}[U,V]) = \Phi + \Psi,
				\end{equation}
				where $\Phi$ is the number of labelled copies of $\tu{J}$ in $\tu{G}$ such that at least one edge of $\tu{J}$ is mapped to some edge with both endpoints in $U$ and $\Psi$ is the number of labelled copies where at least one vertex in $A$ is mapped to some vertex in $V$. 
				
				First we bound $\Phi$. Fix $(a,b)\in E(\tu{J})$ such that $a\in A$ and $b\in B$. Set $\Phi_{a,b}$ to be the number of copies of $\tu{J}$ in $\tu{G}$ such that the edge $(a,b)$ is mapped to an edge with both end points in $U$. Let 
				$\widehat{\tu{J}}$ be as in Lemma \ref{s1:simplelowerbound}(ii). Notice that $\Phi_{a,b} \le |U|^{2}\cdot N(\widehat{\tu{J}}, \tu{G})$.
				Therefore, by Lemma \ref{s1:simplelowerbound}(ii) and \cite[Theorem~5.4]{HMS}, and using $e(\tu{G})\leq Cn^{2}p^{\Delta}$ and $p\ll1 $, we get
				\begin{equation}
					\label{bound cond Phi}
					\Phi \leq \sum_{(a,b) \in E(\tu{J})} \Phi_{a,b}  \leq e_{\tu{J}}\cdot |U|^{2}\cdot \left(2e(\tu{G})\right)^{|A|-1}\cdot n^{|B|-|A|}\leq \frac{e_{\tu{J}}(2C)^{v_{\tu{H}}}}{\eta^{2}}\cdot n^{v_{\tu{J}}}p^{e_{\tu{J}}}\cdot p^{\Delta} \leq \frac{\varepsilon}{2} \cdot n^{v_{\tu{J}}}p^{e_{\tu{J}}}.
				\end{equation}
				We now turn to bound $\Psi$. Fix any $a\in A$. Let $\Psi_{a}$  be the number of labelled copies of $\tu{J}$ in $\tu{G}$ such that the vertex $a$ is mapped to some vertex in $V$. Applying \cite[Lemma~7.1]{BGLZ} to each connected component of $\tu{J}$ we see that there is a $2$-matching $M$ of $\tu{J}$ of size $2|A|$ such that the connected component of $a$ in $M$ is a path. Let $M_{1}, M_{2}, \ldots, M_{s}$ be the components of $M$, where $a\in M_{1}$. Then,
				\begin{equation}
					\label{cond_struc: 2-matching split}
					\Psi_{a}\leq \psi_{a} \cdot \left(\prod_{i=2}^{s} N(M_{s}, \tu{G})\right) \cdot n^{v_{\tu{J}}-\sum_{i=1}^{s} v_{M_{i}}},  
				\end{equation}
				where $\psi_{a}$ is the number of labelled copies of $M_{1}$ with $a$ mapped to a vertex in $V$. 
				
				We claim that we can further guarantee that the components $\{M_s\}_{i=1}^s$ that are paths have odd number of vertices. This follows upon noting that a path with an even number of vertices must have a vertex $u \in A$ with $\deg_{M_i}(u)=1$ and $M_i$ being the connected component of $u$. However, this cannot happen as the degree of any vertex $v\in A$ in any matching of $\tu{J}$ is one. Equipped with this observation, again by an application of \cite[Theorem~5.4]{HMS} we have
				\begin{equation}
					\label{cond_struc: comp bound}
					\prod_{i=2}^{s} N(M_{s}, \tu{G}) \leq \prod_{i=2}^{s} (2e(\tu{G}))^{v_{M_{i}}-\alpha_{M_{i}}^{*}}\cdot n^{2\alpha_{M_{i}}-v_{M_{i}}} \leq (2C)^{v_{\uH}}\cdot n^{\sum_{i=2}^{s}v_{M_{i}}}\cdot p^{(\Delta/2)\sum_{i=2}^{s}e_{M_{i}} },
				\end{equation}  
				where in the last inequality we have used that $2(v_{M_{i}}- \alpha^*_{M_{i}})=e_{M_{i}}$ as $M_{i}$'s are paths with odd number of vertices or cycles with even number of vertices. 
				%
				
				Now, let $v_{M_{1}}=2l+1$ for some $l\in \N$. 
				Let the neighbors of $a$ in $M_{1}$ be $b$ and $c$. Removing the vertices ${a,b,c}$ and all edges incident to them in $M_{1}$ splits $M_{1}$ into at most two disjoint paths $P_{1}$ and $P_{2}$ with $2(k-1)$ and $2(l-k)$ vertices, for some $k\in \set{l}$. So, $\psi_{a}$ can be bounded by appropriately choosing $\{a,b,c\}$ and then choosing $P_1$ and $P_2$. Note that the edge $(a,b)$ has at most $2e(\tu{G})$ choices. Once the image of $a$ is fixed to be some vertex in $V$, $c$ has at most $\eta n$ choices. Therefore, using \cite[Theorem~5.4]{HMS} and $e(\tu{G})\leq Cn^{2}p^{\Delta}$ give us
				\begin{equation}\label{eq:psi-a}
					\psi_{a}\leq 2e(\tu{G})\cdot \eta n\cdot N(P_{1}, \tu{G})\cdot N(P_{2}, \tu{G}) \leq \eta n\cdot (2e(\tu{G}))^{l}\leq \eta (2C)^{v_{\uH}}\cdot  n^{v_{M_1}}p^{\Delta e_{M_1}/2}.
				\end{equation}
				Finally,  upon noting that $\Psi \leq \sum_{a\in A}\Psi_{a}$ and $\sum_{i=1}^{s}e_{M_{i}} = 2|A|= 2e_{\tu{J}}/\Delta$, using the definition of $\eta$ from \eqref{eta:defn}, and combining \eqref{eq:Psi-Phi}-\eqref{eq:psi-a} the proof completes.
			\end{proof}
			
			\begin{remark}\label{rmk:J-copy-bound}
				Observe that for $\tu{H}$ a regular connected graph any $\tu{J} \in {Q}_{\tu{H}}$ is irregular unless $\tu{H}$ is bipartite and $\tu{J}=\tu{H}$. Therefore, for any regular connected $\tu{H}$ Lemma \ref{cond_struc: bipartite lemma} continues to hold for any $\tu{J} \in Q_{\uH}\setminus \{\tu{H}\}$.
			\end{remark}
			
			\section{Proof of Theorems 
				\ref{thm:loc-2} and \ref{cond_struc:star loc-2}}
			As it is clear from the context, throughout this section we use $\UT$ instead of $\text{UT}_{K_{1,r}}(\delta)$. We begin with the proof of Theorem \ref{thm:loc-2}.
			\begin{proof}[Proof of Theorem \ref{thm:loc-2} (Lower Bound)]
				Using binomial tail bounds one can show that the $\log$-probability of the event that the degree of vertex $n$ is greater than $(1+\varepsilon/4)\left\{\delta np^{r}\right\}^{1/r}\cdot n$ is lower bounded by $-(1+\varepsilon/2) \left\{\delta np^{r}\right\}^{1/r}n \log n/r$, for all large $n$. On the other hand, it is straightforward to note that the $\log$-probability of the event that $K_{(1+\varepsilon/4)\lfloor \delta np^r \rfloor, n-1- (1+\varepsilon/4)\lfloor \delta np^r \rfloor} \subseteq \G(n,p)$ is bounded below by $-(1+\varepsilon/2) \lfloor \delta \rho \rfloor n \log n/r$. Using a variance bound and Chebychev's inequality one further finds that the number of copies of $K_{1,r}$ in the rest of graph is at least $(1-\varepsilon\delta/8) n^{r+1}p^r$ with probability bounded away from zero. Upon combining the last three observations we indeed get the lower bound in the regime of $p\in (0,1)$ satisfying $n^{r+1}p^{r}\gg1$ and $np^{r}\rightarrow \rho\in [0,\infty)$,
				\begin{equation}
					\label{s2:Lower bound star}
					\log \P(\UT)\geq \begin{cases}-(1+\varepsilon)\delta^{1/r}\frac{1}{r}n^{1+1/r}p\log n\quad &\text{if $\rho=0$}\\
						-(1+\varepsilon)\left(\left\lfloor\delta \rho\right\rfloor +\left\{\delta \rho\right\}^{1/r}\right)\frac{1}{r} n\log n &\text{if $\rho\in (0,\infty)$}.
					\end{cases}
				\end{equation}
				We skip further details (cf.~\cite[Section 5]{AS25}).
			\end{proof}
			We now move to the proof of the upper bound. This requires some definitions and preparatory results.
			\begin{definition}[Pre-seed graph] Let $\varepsilon>0$ be sufficiently small. Let $\cbar\coloneqq\cbar(\varepsilon, \delta)$ be a sufficiently large constant. A graph $\tu{G}\subseteq K_n$ is said to be a pre-seed graph if the following holds:
				\begin{itemize}
					\item[(PS1)]$\E_{\tu{G}}[N(K_{1,r}, \G(n,p))]\geq (1+\delta(1-\varepsilon))n^{r+1}p^{r}.$ 
					\item[(PS2)] $e(\tu{G})\leq \cbar n^{1+1/r}p\log(1/p)$.
				\end{itemize}
			\end{definition}
			We first show that the probability of $\UT$ is bounded by the existence of pre-seed graphs in $\G(n,p)$. More precisely,
			\begin{lemma}
				\label{upper bound preseed}
				Assume $n^{r+1}p^{r}\gg 1$ and $n^{1/r}p\rightarrow \rho\in [0,\infty)$.
				For large enough $n$, 
				\begin{equation*}
					\P\left(\text{\em{UT}}(\delta)\right) \leq (1+\varepsilon)\cdot\P\left(\G(n,p)\textup{ contains a pre-seed graph}\right).
				\end{equation*}
			\end{lemma}

			\begin{proof}
				
				Using Lemma \cite[Lemma~3.7]{HMS} with $X=N(K_{1,r}, \G(n,p)), d=r, \text{ and }l=\lceil \frac{\cbar}{r} n^{1+1/r}p \log(1/p) \rceil$, for $n$ sufficiently large we get
				\begin{equation}
					\label{mkineq}
					\P\bigg(\UT\cap\{\G(n,p) \text{ contains a pre-seed graph}\}^{c}\bigg) \leq \left(\frac{1+\delta(1-\varepsilon)}{1+\delta} \right)^{\frac{\cbar}{r} n^{1+1/r}p \log(1/p)}.
				\end{equation}
				Therefore taking $\cbar$ sufficiently large (depending only on $\delta$ and $\varepsilon$), as $p \lesssim n^{-1/r}$ we deduce from \eqref{s2:Lower bound star}-\eqref{mkineq} that
				\begin{equation}\label{eq:prob-small}
					\P\big(\UT\cap\{\G(n,p)\text{ contains a pre-seed graph}\}^{c}\big)\leq (\P(\UT))^2.
				\end{equation}
				This completes the proof.
			\end{proof}
			In the next step we consider only those subgraphs for which the property $(\text{PS}1)$ can be replaced by a simpler condition as given below.
			\begin{definition}[Seed graph]
				A graph $\tu{G}\subseteq K_n$ is said to be a seed graph if the following holds:
				\begin{itemize}
					\item[(S1)] $N(K_{1,r}, \tu{G})\geq \delta(1-2\varepsilon)n^{r+1}p^r$,
					\item[(S2)] $e(\tu{G})\leq \cbar n^{1+1/r}p \log(1/p)$.
				\end{itemize}
			\end{definition}

			\begin{lemma}
				\label{preseed to seed}
				Assume $n^{1+1/r}p \geq 1$. For large enough $n$, a pre-seed graph is a seed graph.
			\end{lemma}
			\begin{proof}
				As any subgraph of $K_{1,r}$ without isolated vertices is isomorphic to $K_{1,s}$ for some $s\in \set{r}$, we begin by noting that 
				\begin{equation}
					\label{expectation difference}
					\E_{\tu{G}}\left[N(K_{1,r}, \G(n,p))\right]  \leq\E\left[N(K_{1,r}, \G(n,p))\right]+ \sum_{s=1}^{r-1}{{r\choose s}} N(K_{1,s}, \tu{G}) n^{r-s}p^{r-s} + N(K_{1,r}, \tu{G}),
				\end{equation}
				for any $\tu{G}\subseteq K_n$.  
				%
				By \cite[Theorem~5.4]{HMS} and using the upper bound (PS2) for any $s \in \set{r-1}$ we find that 
				\begin{equation}
					\label{s2: insignificant contribution}
					\begin{aligned}
						N(K_{1,s}, \tu{G}) n^{r-s}p^{r-s} \leq \left( 2\cbar n^{1+1/r}p \log(1/p)\right)^{s} n^{r-s}p^{r-s}\leq \left(2\cbar \right)^{s} \frac{\left(\log(1/p)\right)^{s}}{n^{1-s/r}} n^{r+1}p^{r} = o(n^{r+1}p^{r}), 
					\end{aligned}
				\end{equation}
				for any pre-seed graph $\tu{G}$, where in the last step we used the assumption $n^{1+1/r}p \geq 1$.
				As $\E[N(K_{1,r}, \G(n,p))] = (1+o(1))n^{r+1}p^{r}$, we now conclude from \eqref{expectation difference} and \eqref{s2: insignificant contribution} that any pre-seed graph $\tu{G}$ satisfies $N(K_{1,r},\tu{G}) \geq (1-3\varepsilon) n^{r+1}p^{r}$.
			\end{proof}
			We then show that any seed graph must have a subgraph containing most of its copies of $K_{1,r}$ such that each of the edges of the subgraph participates in a large number of copies of $K_{1,r}$ as well. We call them \textit{core graphs}.  The definition of core graphs below is conceptually similar to that in  \cite{HMS}; however, the conditions (C1) and (C3) are simpler in our setting, as we have already identified the contributing terms in $\E_{\tu{G}}\left[N(K_{1,r}, \G(n,p))\right]$.
			\begin{definition}\label{def:c-graph}
				A graph $\tu{G}\subseteq K_n$ is said to be a core graph if the following holds:
				\begin{itemize}
					\item[(C1)]\label{num copies}$N(K_{1,r}, \tu{G})\geq \delta(1-3\varepsilon)n^{r+1}p^{r}$,
					\item[(C2)] \label{max edges}$e(\tu{G})\leq \cbar n^{1+1/r}p \log(1/p)$, 
				\end{itemize}
				and 
				\begin{itemize}
					\item[(C3)] $\min_{e\in E(\tu{G})}N(K_{1,r}, \tu{G},e)\geq \delta \varepsilon n^{r+1}p^{r}/ \left(\cbar n^{1+1/r}p \log(1/p)\right)$.
				\end{itemize}
			\end{definition}
			Once we have a seed graph $\tu{G}$ then we can remove its edges iteratively that participate in strictly less that $\delta\varepsilon n^{r+1}p^{r}/\left(\cbar n^{1+1/r}p \log(1/p)\right)$ labelled copies of $K_{1,r}$ to produce a subgraph $\tu{G}^{*}$ so that condition (C3) is satisfied. 
			Therefore $\tu{G}^{*}$ is a core graph.
			So, Lemmas \ref{upper bound preseed} and \ref{preseed to seed} along with the above discussion gives us the following result.
			\begin{proposition}
				\label{1st reduction}
				Assume $n^{r+1}p^{r}\gg 1$ and $n^{1/r}p\rightarrow \rho \in [0,\infty)$. For large enough $n$,
				\begin{equation*}
					\P\left(\text{\em{UT}}(\delta)\right) \leq (1+\varepsilon)\P\left(\G(n,p)\textup{ contains a core graph}\right).
				\end{equation*}
			\end{proposition}
			Armed with Proposition \ref{1st reduction}, the proof of the upper bound reduces to the following proposition.
			\begin{proposition}
				\label{main proposition}
				Assume $(\log n)^{1/(r-1)}\ll n^{1+1/r}p$. If $np^{r}\ll 1$, then for large enough $n$, we have that
				\begin{equation*}
					\P\left(\G(n,p)\textup{ contains a core graph}\right)\leq \exp\left(-(1-\mathfrak{f}(\varepsilon))\frac{\delta^{1/r}}{r}n^{1+1/r}p\log n\right),
				\end{equation*}
				and if $np^{r}\rightarrow \rho\in (0,\infty)$, then for large enough $n$, we have
				\begin{equation}\label{eq:p-sim-bdry}
					\P\left(\G(n,p)\textup{ contains a core graph}\right) \leq \exp\left(-(1-\mathfrak{f}(\varepsilon))\left(\lfloor \delta \rho\rfloor+\left\{\delta \rho\right\}^{1/r}\right)\frac{1}{r}n\log n\right),
				\end{equation}
				for some nonnegative function $\mathfrak{f}(\cdot)$ such that $\lim_{\varepsilon\downarrow 0}\mathfrak{f}(\varepsilon)=0$.
			\end{proposition}
			The key idea is to show that the subgraph of a core graph induced by edges incident to vertices of low degree is bipartite, and then use a combinatorial argument to show that such a family of core graphs is entropically stable, in the sense that the logarithm of the number of such cores is negligible compared to the $\log$-probability cost of realizing a core. This ensures that a union bound over all such cores preserves not only the correct speed of decay but also the constant in front of it.
			To make it precise, let $\tu{G}$ be a core graph and we consider the following set of low degree vertices
			\begin{equation*}
				\W(\tu{G}) \coloneqq \left\{v\in V(\tu{G}):\text{deg}_{\tu{G}}(v)\leq \frac{1}{\varepsilon}\right\}.
			\end{equation*}
			Let $\tu{G}_{\W}$ be the subgraph induced by edges that are incident to some vertex in $\W$. We will show that there are no edges between vertices in $\W$. For this we need to obtain a lower bound on the product of degrees of the endpoints of edges in core graphs.

			\begin{lemma}
				\label{product of degreess}
				Let $\tu{G}$ be a core graph. If $n^{1+1/r}p\geq 1$ then for every edge $e=(u,v)\in E(\tu{G})$
				\begin{equation}
					\label{product of deg:weak bound}
					\textup{deg}_{\tu{G}}(u)\cdot \textup{deg}_{\tu{G}}(v)\geq \frac{\widetilde{c}_{0}n^{1+1/r}p}{\left(\log n\right)^{1/(r-1)}},
				\end{equation}
				for some constant {$\widetilde{c}_{0}\coloneqq \widetilde{c}_{0}(\varepsilon, \delta)>0$}.
			\end{lemma}
			\begin{proof}
				Any copy of $K_{1,r}$ that uses an edge $e=(u,v)$ of graph $\tu{G}$ must have either $u$ or $v$ as its center vertex. Then one of the $r$ edges of $K_{1,r}$ is mapped to $e=(u,v)$ and the rest $(r-1)$ vertices are mapped to the neighbors of either $u$ or $v$ depending on which is the center vertex of the copy. Therefore, 
				\begin{equation}
					\label{local embedding bound}
					\begin{aligned}
						N(K_{1,r},\tu{G},e) = r\cdot\prod_{i=1}^{r-1}\left({\text{deg}_{\tu{G}}(u)-i}\right)+r\cdot\prod_{i=1}^{r-1}\left({\text{deg}_{\tu{G}}(v)-i}\right)\leq r\cdot \left(\text{deg}_{\tu{G}}(u)+\text{deg}_{\tu{G}}(v)\right)^{r-1}.
					\end{aligned}
				\end{equation}
				$\tu{G}$ being a core graph satisfies condition (C3). Therefore, together with \eqref{local embedding bound}, we have
				\begin{equation*}
					\text{deg}_{\tu{G}}(u)+\text{deg}_{\tu{G}}(v) \geq \left(\frac{\delta\varepsilon}{\cbar(r+1)}\right)^{1/(r-1)}\frac{n^{1+1/r}p}{\left(\log n\right)^{1/(r-1)}},
				\end{equation*}
				where we also used $n^{1+1/r}p\geq 1$. Observe that 
				$\max\{\text{deg}_{\tu{G}}(u),\text{deg}_{\tu{G}}(v)\}\geq \frac{1}{2} \left(\text{deg}_{\tu{G}}(u)+\text{deg}_{\tu{G}}(v)\right)$ and  $\min\{\text{deg}_{\tu{G}}(u),\text{deg}_{\tu{G}}(v)\}\geq 1$. This completes the proof.
			\end{proof}
			
			For the lower bound on the product of degrees to be useful we will assume $n^{1+1/r}p\gg \left(\log n\right)^{1/(r-1)}$. Let $u\in \W\subset V(\tu{G})$ and $(u,v)\in E(\tu{G})$, where $\tu{G}$ is a core graph. As $n^{1+1/r}p \gg (\log n)^{1/(r-1)}$, Lemma \ref{product of degreess} implies that
			\begin{equation}
				\label{bipartite argument}
				\text{deg}_{\tu{G}}(v)\geq \varepsilon\cdot \text{deg}_{\tu{G}}(u)\cdot \text{deg}_{\tu{G}}(v)\geq \varepsilon \frac{\widetilde{c}_{0}n^{1+1/r}p}{(\log n)^{1/(r-1)}} \geq \frac{2}{\varepsilon},
			\end{equation}
			for large enough $n$. This lower bound implies that $v\notin \W$. Hence $\tu{G}_{\W}$ is bipartite.\\

			In the next lemma using the above property of $\tu{G}_{\W}$ we derive a bound on the number of core graphs in terms of its number of edges and number of vertices of small degree.
			\begin{lemma}
				\label{counting cores}
				Fix $\varepsilon\in (0,1/2)$ and $r\geq 2$. 
				Let $\mathcal{N}(\mathbi{e}, \textbf{w}, \varepsilon)$ be the number of core graphs with $e(\tu{G})=\mathbi{e}$ and $|\W(\tu{G})|=\textbf{w}$. Then for any $p\in (0,1)$ satisfying $n^{-1-1/r}(\log n)^{1/(r-1)}\ll p\lesssim n^{-1/r}$ and large enough $n$,
				\begin{equation*}
					\mathcal{N}(\mathbi{e}, \textbf{w}, \varepsilon)\leq {n\choose \textbf{w}}\cdot \exp\left(17r\varepsilon \mathbi{e}\log(1/p)\right).
				\end{equation*}
			\end{lemma}
		\begin{proof}
			The proof uses the same general strategy as that of \cite[Lemma~4.7]{ABRB}, particularly the ideas used to handle the regime where $p$ is not too sparse. We briefly outline the main approach. Let $\tu{G}$ be any core graph with $e(\tu{G})=\mathbi{e}$ and $|\W(\tu{G})|=\mathbi{w}$. We start by splitting the vertices in $V(\tu{G})\setminus \W(\tu{G})$ as follows: for each $k=1,\ldots,\lceil\log_{2}(\varepsilon n)\rceil +1$, define 
			\begin{equation*}
				V_{k}\coloneqq \{v\in V(\tu{G}):\deg_{\tu{G}}(v)\geq 2^{k-1}\cdot\varepsilon^{-1}\}.
			\end{equation*} 
		This forms a nested sequence of subsets: $\phi=V_{\lceil \log_{2}(\varepsilon n)\rceil+1}\subset\cdots\subset V_{1}= V(\tu{G})\setminus \W(\tu{G})$. 
		
		Thanks to \eqref{product of deg:weak bound}, for any edge $(u,v)\in E(\tu{G})$ with $u\in V_{i_0-1}\setminus V_{i_0}$ for some $i_0\geq 2$, we have $v\in V_{j_0}$, where $j_{0}\coloneqq j_{0}(i_0)$ is the smallest integer satisfying 
		\[2^{i_0+j_0-1}\varepsilon^{-2}\geq \frac{\widetilde{c}_0 n^{1+1/r}}{(\log n)^{1/(r-1)}}.\]
		Moreover, if $(u,v)\in E(\tu{G})$ with $u\in \W(\tu{G})$, then $v\in V_{j_{*}}$, where $j_{*}$ is the smallest integer satisfying 
		\[2^{j_*}\varepsilon^{-2}\geq \frac{\widetilde{c}_{0}n^{1+1/r}p}{(\log n)^{1/(r-1)}}.\]
		These structural constraints allow us to bound the number of ways to construct a core $\tu{G}$ with $e(\tu{G})= \mathbi{e}$ and $|\W(\tu{G})|=\mathbi{w}$ by following the same argument used to obtain \cite[(4.20)--(4.22)]{ABRB}. We omit the details, as the remaining computations are identical.
		\end{proof}
			{We need another improvement of the bound on the number of copies of $K_{1,r}$ in a graph $\tu{G}$.}
			\begin{lemma}
				\label{global embedding bound}
				Assume $t\geq 2$.
				For every graph $\tu{G}\subseteq K_n$, 
				\begin{equation*}
					N(K_{1,t}, \tu{G})\leq \left(e(\tu{G})\right)^{t}.
				\end{equation*}
			\end{lemma}
			Note that Lemma \ref{global embedding bound} improves \cite[Theorem~5.7]{HMS} by a factor of $2^t$. This will be essential in deriving the rate function of the upper tail event when $np^{r}\ll 1$.   
			\begin{proof}
				Let $v\in V(\tu{G})$, then the number of labeled copies of $K_{1,t}$ with center at $v$ is $\prod_{i=0}^{t-1}(\text{deg}_{\tu{G}}(v)-i)$. Since two copies of $K_{1,t}$ with different centers have to be distinct, we get the following result,
				\begin{equation*}
					N(K_{1,t}, \tu{G})=\sum_{v\in V(\tu{G})}\prod_{i=0}^{t-1}\left(\text{deg}_{\tu{G}}(v)-i\right)\leq \left(e(\tu{G})\right)^{t-2} \sum_{v\in \tu{G}}\text{deg}_{\tu{G}}(v)(\text{deg}_{\tu{G}}(v)-1) = \left(e(\tu{G})\right)^{t-2}N(K_{1,2},\tu{G}),
				\end{equation*}
				where in the inequality we used that $\text{deg}_{\tu{G}}(v)\leq e(\tu{G})$ and in the last step we used the formula from the first equality with $t=2$. Therefore, it is enough to prove the lemma for $t=2$.
				Let the adjacency matrix of $\tu{G}$ be denoted by $A=(a_{i,j})_{1\leq i,j\leq n}$. Then
				\begin{equation}
					\label{eq15}
					4\left(e(\tu{G})\right)^{2} = \left(\sum_{i,j=1}^{n}a_{i,j}\right)^{2} = \sum_{i_1,i_2,i_3,i_4=1}^{n}a_{i_1,i_2}a_{i_3,i_4} \ge \sum_{m_1,m_2=1}^2 \sum_{(i_1,i_2,i_3,i_4) \in A_{m_1,m_2}} a_{i_1,i_2}a_{i_3,i_4}, 
				\end{equation}
				where for  $m_1, m_2 \in \{1,2\}$ we set 
				\[
				A_{m_1,m_2} \coloneqq \{(i_1,i_2,i_3,i_4)\in \set{n}^4: i_{m_1}=i_{m_2+2} \text{ and } \{i_m\}_{m=1}^4\setminus\{i_{m_2+2}\} \text{ are pairwise distinct}\}. 
				\]
				Observe that for any $m_1$ and $m_2$ as above we have $\sum_{(i_1,i_2,i_3,i_4) \in A_{m_1,m_2}} a_{i_1,i_2}a_{i_3,i_4} = N(K_{1,2}, \tu{G})$. Hence, by \eqref{eq15} the claimed bound is immediate.
			\end{proof}


		To get the correct rate function when $np^{r}\rightarrow \rho\in (0,\infty)$, we require a slightly more precise lower bound satisfied by the number of edges of a core graph. 
		\begin{lemma}
			\label{s2: improved geb }
			Let $\tu{G}$ be a core graph on $n$ vertices. Assume $np^{r}\rightarrow \rho\in (0,\infty)$. Then,
			\begin{equation*}
				e(\tu{G})\geq (1-\varepsilon)\left(\lfloor\delta\rho(1-5\varepsilon)\rfloor + \{\delta\rho(1-5\varepsilon)\}^{1/r}\right) \cdot n.
			\end{equation*}
		\end{lemma}
		\begin{proof}
			Assume the contrary. Thus, $e(\tu{G})< C_{\delta, \rho} n \lesssim n^{2}p^{r}$, 
			for some constant $C_{\delta, \rho}$ depending only on $\delta$ and $\rho$. Then, by 
			Lemma \ref{cond_struc: bipartite lemma}, there exists a partition $V(\tu{G})= A\cup B$ such that  
			\begin{equation}
				\label{s2: eq frac}
				N_{A}(K_{1,r}, \tu{G}[A,B])\geq N(K_{1,r}, \tu{G})-\delta\varepsilon n^{r+1}p^{r},
			\end{equation}
			where 
			$N_{A}(K_{1,r}, \tu{G}[A,B])$ is the number of labelled copies of $K_{1,r}$ in $\tu{G}[A,B]$ where the center vertex is mapped to a vertex in $A$. 
			Since $\tu{G}$ is a core graph, using condition (C1) in \eqref{s2: eq frac} and {\cite[Lemma~5.14(i)]{HMS}} we have
			\begin{equation*}
				\left(\left \lfloor\frac{e\left(\tu{G}[A,B]\right)}{|B|}\right\rfloor +\left \{\frac{e\left(\tu{G}[A,B]\right)}{|B|}\right\}^{r}\right)\cdot  |B|^{r} \geq  N_{A}(K_{1,r},\tu{G}[A,B] )\geq \delta(1-4\varepsilon)n^{r+1}p^{r}.
			\end{equation*}
			Further note that by definition of $A$ in Lemma \ref{cond_struc: bipartite lemma}, we have $|A|\lesssim e(\tu{G})/n  \lesssim 1$. Therefore, using $(1-\varepsilon)n \leq |B|\leq n$ and solving for $e(\tu{G}[A,B])$, we get
			\begin{equation*}
				e(\tu{G}[A,B])\geq (1-\varepsilon)\left(\left\lfloor\delta\rho(1-5\varepsilon)\right\rfloor + \left\{\delta\rho(1-5\varepsilon)\right\}^{1/r}\right)\cdot n,
			\end{equation*}
			which contradicts our assumption.
		\end{proof}

		We now finish the proof of the upper bound.
		\begin{proof}[Proof of Proposition \ref{main proposition}]
			By Lemma \ref{global embedding bound} for any core graph $\tu{G}$ 
			\begin{equation*}
				e_{\min} \coloneqq \delta^{1/r}(1-3\varepsilon)^{1/r}n^{1+1/r}p \leq e(\tu{G})\leq \cbar n^{1+1/r}p\log(1/p) \eqqcolon e_{\max}.
			\end{equation*}
			As $\tu{G}_{\mathcal{W}}$ is bipartite we also note that the set $\W(\tu{G})$ can have at most $e(\tu{G})$ elements.
			Fix any $e_{\min}\leq \mathbi{e} \leq e_{\max}$ and $\mathbi{w}\leq \mathbi{e}$. We claim that 
			\begin{multline}
				\label{main proof eq 1}
				\P\left(\exists \tu{G}\subset \G(n,p):\ e(\tu{G})=\mathbi{e}\text{ and } |\W(\tu{G})|=\mathbi{w}\right)\leq {n\choose \mathbi{w}}\cdot p^{-17r\varepsilon \mathbi{e}}\cdot p^{\mathbi{e}} \\
				\leq \exp\left(-(1-36r^{2}\varepsilon)(1-3\varepsilon)^{1/r}\frac{\delta^{1/r}}{r}n^{1+1/r}p\log n\right).
			\end{multline}
			The first inequality is immediate from Lemma \ref{counting cores}. To prove the second inequality we split the range of  ${\bm w}$ and ${\bm e}$ into two cases: $\mathbi{w}\leq \varepsilon \mathbi{e}$ and $\mathbi{w}\geq \varepsilon \mathbi{e} $. In the first case, we have
			\begin{equation}
				\label{main proof eq 2}
				{n\choose \mathbi{w}}\cdot p^{(1-17r\varepsilon)\mathbi{e}}\leq n^{\mathbi{w}}\cdot p^{(1-17r\varepsilon)\mathbi{e}}\leq n^{\varepsilon \mathbi{e}}\cdot p^{(1-17r\varepsilon)\mathbi{e}} \leq n^{-(1-19r\varepsilon)\mathbi{e}/r},
			\end{equation}
			where in the last step we used $p\lesssim  n^{-1/r}$.
			In the second case, using the lower bound $\mathbi{w}\geq \varepsilon \mathbi{e}\geq \varepsilon e_{\min}$ we obtain that, for all large $n$
			\begin{equation*}
				{n\choose \mathbi{w}} \cdot p^{(1-17r\varepsilon)\mathbi{e}} \leq \left(\frac{en}{\mathbi{w}}\right)^{\mathbi{w}} \cdot p^{(1-17r\varepsilon)\mathbi{e}}\leq \left(\frac{e}{\delta^{1/r}(1-3\varepsilon)^{1/r}}\right)^{\mathbi{w}}\cdot n^{-\mathbi{w}/r}\cdot p^{(1-17r\varepsilon)\mathbi{e}-\mathbi{w}}\leq n^{-\mathbi{w}/r} \cdot p^{(1-18r\varepsilon)\mathbi{e}-\mathbi{w}},
			\end{equation*}
			where in the last step we have used $\mathbi{w}\leq \mathbi{e}$ and $p\ll1$. 
			To complete the proof of \eqref{main proof eq 1} we further split into two sub cases: $\mathbi{w}\leq (1-18r\varepsilon)\mathbi{e}$ and $\mathbi{w}\geq (1-18r\varepsilon)\mathbi{e}$. If $\mathbi{w}\leq (1-18r\varepsilon)\mathbi{e}$, then using $p\lesssim n^{-1/r}$ we get
			\begin{equation}
				\label{main proof eq 3}
				n^{-\mathbi{w}/r}\cdot p^{(1-18r\varepsilon)\mathbi{e}-\mathbi{w}}\leq n^{-(1-19r\varepsilon){\bm e}/r} .
			\end{equation}
			On the other hand, if $\mathbi{w}\geq (1-18r\varepsilon)\mathbi{e}$, since we also have $\mathbi{w}\leq \mathbi{e}$, we get
			\begin{equation}
				\label{main proof eq 4}
				n^{-\mathbi{w}/r}\cdot p^{(1-18r\varepsilon)\mathbi{e}-\mathbi{w}}\leq n^{-(1-18r\varepsilon)\mathbi{e}/r}\cdot p^{-18 r\varepsilon \mathbi{e}}\leq n^{-(1-36r^{2}\varepsilon)\mathbi{e}/r},
			\end{equation}
			where in the last step we used $n^{1+1/r}p\geq 1$. 
			Combining \eqref{main proof eq 2}-\eqref{main proof eq 4} along with $\mathbi{e}\geq e_{\min}$  we obtain the second inequality in \eqref{main proof eq 1}.

			Now, summing up both sides of \eqref{main proof eq 1} over allowable range of $\mathbi{e}$ and $\mathbi{w}$, as $n^{1+1/r} p \gg 1$, we derive that

			\begin{equation}\label{s2: bound deltanprlessthan1}
				\P\left(\bigcup_{\mathbi{e}, \mathbi{w}} \left\{\exists \tu{G}\subset \G(n,p):\ e(\tu{G})=\mathbi{e}, |\W(\tu{G})|=\mathbi{w}\right\}\right)\\[-1pt]
				\leq \exp\left(-(1-37r^{2}\varepsilon)(1-3\varepsilon)^{1/r}\frac{\delta^{1/r}}{r}n^{1+1/r}p\log n\right),
			\end{equation}
			for all large $n$. See that this bound \eqref{s2: bound deltanprlessthan1} is optimal if $\delta np^{r}<1$. Now, assume $np^{r}\rightarrow \rho\in[1,\infty)$. In this case, we use Lemma \ref{s2: improved geb } to get that 
			$
			e(\tu{G})\geq \widehat e_{\min} \coloneqq (1-\varepsilon)\left(\left\lfloor\delta\rho(1-5\varepsilon)\right\rfloor+\left\{\delta\rho(1-5\varepsilon)\right\}^{1/r}\right)\cdot n
			$. 
			Using the lower bound $\widehat{e}_{\min}$ on the edges instead of $e_{\min}$ and repeating the same steps as in the proof of \eqref{s2: bound deltanprlessthan1} we derive \eqref{eq:p-sim-bdry}. 
		This completes the proof.
	\end{proof}

	Next we turn our focus to proving Theorem \ref{cond_struc:star loc-2}. We use the concept of \textit{strong-core} graphs. 
	
	\begin{definition}\label{def:sc-graph}
		Let $\cstar\coloneqq \cstar(\delta, \varepsilon,r)>0$ be some constant, which we will fix later. We define a graph $\tu{G}\subseteq K_{n}$ to be a \textit{strong-core} graph if the following holds:
		\begin{itemize}
			\item[(SC1)] $N(K_{1,r}, \tu{G})\geq \delta(1-4\varepsilon) n^{r+1}p^{r}$,
			\item[(SC2)]$e(\tu{G})\leq \cstar n^{1+1/r}p$,
		\end{itemize}
		and
		\begin{itemize}
			\item[(SC3)]$\min_{e\in E(\tu{G})} N(K_{1,r}, \tu{G}, e)\geq (\delta \varepsilon/\cstar)\cdot (n^{1+1/r}p)^{r-1}$. 
		\end{itemize}
	\end{definition}
	
	As will be seen in the lemma below the advantage of a strong-core graph over a core graph is that if 
	$\tu{G}$ is strong-core graph then it has a further subgraph $\glow$ which satisfies a strong upper and lower bound on product of degrees of adjacent vertices and contains almost the same number of labelled copies of $K_{1,r}$ as in $\tu{G}$. This will be essential to our proof.

	\begin{lemma}
		\label{glow properties}
		Let $\tu{G}$ be a strong-core graph. 
		The following hold:
		\begin{itemize}
			\item[(i)] There exists some constant $c_{0}\coloneqq c_{0}(\varepsilon, \delta, r)>0$ such that for every edge $e=(u,v)\in E(\tu{G})$
			\begin{equation}
				\label{product of deg:strong bound}
				\tu{deg}_{\tu{G}}(u)\cdot \tu{deg}_{\tu{G}}(v) \geq c_{0}n^{1+1/r}p.
			\end{equation}
			\item[(ii)] There exists a large constant $C_{0}\coloneqq C_{0}(\varepsilon, \delta, r) <\infty$ such that the subgraph $\glow\subseteq \tu{G}$ spanned by edges $e=(u,v)\in E(\tu{G})$ for which
			\begin{equation}
				\label{product of deg:upper bound}
				\tu{deg}_{\tu{G}}(u)\cdot \tu{deg}_{\tu{G}}(v) \leq C_{0}n^{1+1/r}p, 
			\end{equation}
			satisfies
			\begin{equation}
				\label{glow contains almost all copies}
				N(K_{1,r}, \glow)\geq \delta(1-5\varepsilon)n^{r+1}p^{r}.
			\end{equation}
		\end{itemize}
	\end{lemma}
	Lemma \ref{glow properties} will be proved later. For now, we prove Theorem \ref{cond_struc:star loc-2} by leveraging Lemma \ref{glow properties}.
	
	\begin{proof}[Proof of Theorem \ref{cond_struc:star loc-2}]
		Fix $\chi \in (0,1)$.
		Let 
		$np^{r}\ll1$. By \eqref{eq:prob-small} and arguing similarly as in the proof of Proposition \ref{1st reduction} 
		\begin{equation}
			\label{cond_struc: star step 1}
			\P\left(\UT\cap\left\{ \G(n,p)\tu{ does not contain a core graph} \right\}\right) \leq  (\P(\UT))^2,
		\end{equation}
		for any $\varepsilon \in (0,1)$. Next observe that by repeating the same steps as in the proof of \eqref{s2: bound deltanprlessthan1} with $e(\tu{G})\geq (1+C r^{2}\varepsilon) \delta^{1/r}n^{1+1/r}p$, we obtain for sufficiently small $\varepsilon>0$ and some large absolute constant $C < \infty$,
		\begin{equation}
			\label{cond_struc: star step 2}
			\P\left(\{\G(n,p)\tu{ contains a core graph $\widetilde{\tu{G}}$, with } e(\widetilde{\tu{G}})\geq (1+C r^{2}\varepsilon)\delta^{1/r}n^{1+1/r}p\}\right) \leq  (\P(\UT))^{1+2\varepsilon}.
		\end{equation}
		Set $\cstar = (1+C r^{2}\varepsilon)\delta^{1/r}$. We claim that any core graph $\widetilde{\tu{G}}$ with $e(\widetilde{\tu{G}})\leq \cstar n^{1+1/r}p$ contains a strong-core subgraph $\tu{G}\subset \widetilde{\tu{G}}$. To see this, we iteratively remove edges from a core graph $\widetilde{\tu{G}}$ that participate in less than $(\delta \varepsilon/\cstar)\cdot (n^{1+1/r}p)^{r-1}$ copies of $K_{1,r}$. This process yields a subgraph $\tu{G}$ satisfying both (SC2) and (SC3). Property (SC1) follows directly by triangle inequality. Therefore, from \eqref{cond_struc: star step 1} and \eqref{cond_struc: star step 2} we have
		\begin{equation}
			\label{reduction to strong core graphs}
			\P\left(\left\{\G(n,p)\tu{ contains a strong-core graph G, with }e(\tu{G})\leq \cstar n^{1+1/r}p\right\}\Big| \UT \right)\geq 1- (\P(\UT))^{\varepsilon},
		\end{equation}
		for any $\varepsilon >0$. To complete the proof we will show that any strong-core graph $\tu{G}$ with $e(\tu{G}) \leq \overline{C}_*n^{1+1/r} p$ has at least one vertex with degree at least $(1-\chi)n^{1+1/r} p$, upon choosing $\varepsilon >0$ depending on $ \chi$. 
		
		Turning to prove the above take $\tu{G}\subseteq K_{n}$ to be any strong-core graph. By Lemma \ref{glow properties} there exists a subgraph $\glow\subseteq \tu{G}$ with properties \eqref{product of deg:strong bound},\eqref{product of deg:upper bound}, and \eqref{glow contains almost all copies}.
		Now fix $\gamma\coloneqq C_{0} \left(2\cstar/\delta\varepsilon\right)^{1/(r-1)}$.  Define 
		\[
		U\coloneqq \{v\in V(\glow): \tu{deg}_{\tu{G}}(v)\leq \gamma \} \text{ and } V\coloneqq\{v\in V(\glow): (u,v)\in E(\glow)\tu{ for some }u\in U \}. 
		\]
		Observe that there cannot be any edge in $\glow$ with both endpoints in $U$, as this would contradict \eqref{product of deg:strong bound} for sufficiently large $n$.  Thus, $U\cap V = \phi$. Consequently, we obtain the following decomposition:
		\begin{equation}
			\label{glow_split}
			N(K_{1,r}, \glow) = N_{V}(K_{1,r}, \glow[U,V]) + \mathfrak{R},
		\end{equation}  
		where $\mathfrak{R}$ is the number of labelled copies of $K_{1,r}$ in $\glow$ such that a leaf vertex of $K_{1,r}$ is mapped to a vertex outside $U$. We aim to show that $\mathfrak{R}$ is negligible.
		To this end, note that if $(x,y)\in E(\glow)$ such that $x\notin U$, then by \eqref{product of deg:upper bound} we have
		\begin{equation}
			\label{deg upper bound}
			\tu{deg}_{\tu{G}}(y)\leq \frac{1}{\gamma}\tu{deg}_{\tu{G}}(x)\cdot\tu{deg}_{\tu{G}}(y) \leq\frac{C_{0}}{\gamma} \cdot n^{1+1/r}p\leq \left(\frac{\delta \varepsilon}{2\cstar}\right)^{1/(r-1)} n^{1+1/r}p.
		\end{equation} 
		Thus, for any copy of $K_{1,r}$ in $\glow$ that has one leaf vertex outside $U$ must have its centre vertex satisfy \eqref{deg upper bound}. Hence 
		\begin{equation}
			\label{stars with one leaf outside U}
			\mathfrak{R}\leq \sum_{w\in V(\glow)\setminus U} \left(\tu{deg}_{\tu{G}}(w)\right)^{r}\leq 2e(\glow)\cdot \frac{\delta\varepsilon}{2\cstar}\left(n^{1+1/r}p\right)^{r-1}\leq \delta \varepsilon n^{r+1}p^{r}. 
		\end{equation}
		Let $\Delta_{n}\coloneqq \max_{u\in \set{n}}\tu{deg}_{\G(n,p)}(u)$.  Plugging the bound \eqref{stars with one leaf outside U} in \eqref{glow_split}, and using \eqref{glow contains almost all copies} we get
		\begin{equation*}
			\delta(1-6\varepsilon)n^{r+1}p^{r} \leq N_{V}(K_{1,r}, \glow[U,V]) \leq \sum_{v\in V} (\tu{deg}_{\tu{G}}(v))^{r}\leq e(\glow[U,V])\cdot \Delta_{n}^{r-1}\leq \cstar n^{1+1/r}p \cdot \Delta_{n}^{r-1}.
		\end{equation*}
		Using the definition of $\cstar$, we conclude that for sufficiently small $\varepsilon>0$, $\Delta_{n}\geq (1-\chi)n^{1+1/r}p$. Combining this with \eqref{reduction to strong core graphs}, the proof of part (a) is complete.

		We now turn to proving part (b). Assume $np^{r}\rightarrow \rho\in (0,\infty)$. Set $q\coloneqq \lfloor \delta \rho\rfloor + \{\delta\rho\}^{1/r}$.
		Arguing in the same way as \eqref{cond_struc: star step 1}-\eqref{cond_struc: star step 2}, we have
		\begin{equation*}
			\P\left(\left\{ \G(n,p)\tu{ contains a core graph G with }e(\tu{G})\leq (1+\mathfrak{s}_{0}(\varepsilon))\cdot qn\right\}\mid \UT\right)\geq 1-(\P(\UT))^\varepsilon,
		\end{equation*}
		for some non-negative function $\mathfrak{s}_{0}(\cdot)\coloneqq \mathfrak{s}_{0}(\cdot, \delta, r)$ such that $\lim_{\varepsilon\downarrow 0}\mathfrak{s}_{0}(\varepsilon)=0$. Now let $\tu{G}$ be a core graph with $e(\tu{G})\leq (1+\mathfrak{s}_0(\varepsilon))\cdot qn$.
		By Lemma \ref{cond_struc: bipartite lemma} there exists a partition $V(\tu{G})=U\cup V$ with $|U|\lesssim e(\tu{G})/n \asymp 1$ such that 
		\begin{equation*}
			N_{U}(K_{1,r}, \tu{G}[U,V])\geq \delta(1-4\varepsilon)n^{r+1}p^{r}.
		\end{equation*}
		Therefore using \cite[Lemma~5.14(ii)]{HMS} with $\tu{G}[U,V]$, we obtain a subset $W\subseteq U$ of size $\lceil \delta \rho\rceil$ such that 
		\begin{equation*}
			e(\tu{G}[W, V])\geq (1-\mathfrak{s}_{1}(\varepsilon))\cdot qn,
		\end{equation*}
		for some non-negative function $\mathfrak{s}_{1}(\cdot)\coloneqq \mathfrak{s}_{1}(\cdot, \delta, r)$ such that $\lim_{\varepsilon\downarrow 0}\mathfrak{s}_{1}(\varepsilon)= 0$. There further exists a subset $W^{'}\subset W$ of size $\lfloor \delta \rho\rfloor$ such that $\tu{deg}_{\tu{G}}(w)\geq (1-\mathfrak{s}_{1}(\varepsilon))n$ for every $w\in W^{'}$. Thus, only in the case $\delta \rho \notin \N$ the set $W\setminus W' \neq \phi$ and has size one. In that case the only vertex in $W\setminus W^{'}$ has degree at least $(1-\mathfrak{s}_{2}(\varepsilon))\{\delta\rho\}^{1/r}n$ for some non-negative function $\mathfrak{s}_{2}(\cdot)\coloneqq \mathfrak{s}_{2}(\cdot, \delta, \rho, r)$ such that $\lim_{\varepsilon\downarrow 0}\mathfrak{s}_{2}(\varepsilon)= 0$. This completes the proof.
	\end{proof}
	
	Let us now prove Lemma \ref{glow properties}. Its proof is motivated from that of \cite[Lemma~4.2]{ABRB}. 
	\begin{proof}[Proof of Lemma \ref{glow properties}]
		The argument for part (i) is analogous to the proof of Lemma \ref{product of degreess}, with the modification that the graph $\tu{G}$ satisfies the stronger condition (SC3) rather than (C3).
		
		For part (ii), take $C_{0}=5r\cdot 2^{r-1}\cstar^{r+1}/(\delta\varepsilon)$ and let $\glow$ be as defined in Lemma \ref{glow properties}.  Take $\ghigh$ to be the complement of the subgraph $\glow$ in $\tu{G}$.  We claim that 
		\begin{equation}
			\label{bound on ghigh}
			e(\ghigh)\leq  \frac{\delta\varepsilon}{r\cdot (2\cstar)^{r-1}} \cdot n^{1+1/r}p.
		\end{equation} To observe this, note that we have $N(P_{4}, \tu{G})\leq (2e(\tu{G}))^{2}$, where $P_{4}$ is a path on $4$ vertices. Simultaneously, we also have 
		\begin{equation*}
			\begin{aligned}
				N(P_{4}, \tu{G}) &\geq 2 \sum_{(u,v) \in E(\tu{G})} (\deg_\tu{G}(u)-1) \cdot (\deg_\tu{G}(v) -2) \\[-6pt]
				&\geq  2 \sum_{(u,v) \in E(\tu{G})} \deg_\tu{G}(u) \cdot \deg_\tu{G}(v) - 3 \sum_{v \in V(\tu{G})} (\deg_\tu{G}(v))^2 
				\ge 2 C_0 n^{1+1/r} p \cdot e(\ghigh)  - 6 e(\tu{G})^2.
			\end{aligned}
		\end{equation*}
		Combining the upper and lower bounds on $N(P_{4}, \tu{G})$ and using (SC2) gives us the claim \eqref{bound on ghigh}. The inequality in \eqref{local embedding bound} shows that
		\begin{equation*}
			\begin{aligned}
				N(K_{1,r}, \tu{G})-N(K_{1,r}, \glow)&\leq \sum_{(u,v)\in E(\ghigh)}N(K_{1,r}, \tu{G}, uv)\\ &\leq r \sum_{(u,v)\in E(\ghigh)}  \left(\tu{deg}_{\tu{G}}(u)+\tu{deg}_{\tu{G}}(v)\right)^{r-1}\leq r\cdot(2e(\tu{G}))^{r-1}\cdot e(\ghigh).
			\end{aligned}
		\end{equation*}
		Finally the desired conclusion is obtained by using the bound \eqref{bound on ghigh} and (SC1).
	\end{proof}

	\section{Proof of Theorem \ref{thm:mean-field}}
	\label{sec:mean_field}
	We remind the reader that for any ${\bm\xi} = (\xi_{i,j})_{1\leq i<j\leq n}$, in the notation $\E_{\mu_{\bm \xi}}[N(K_{1,r}, \G_n)]$ (and similarly in $\tu{Var}_{\mu_{\bm \xi}}(N(K_{1,r}, \G_n))$) the subscript $\mu_{\bm\xi}$ denotes that the edges of random graph $\G_{n}$ are independent, with the edge between the pair of vertices $i \ne j \in [n]$ being present with probability $\xi_{i,j}$. 
	We now turn to the key ingredient of the proof:
%
	\begin{equation}\label{eq:var-bd}
		\sup_{\boldsymbol{\xi}\in \tu{S}_{\delta, \varepsilon}}\frac{\tu{Var}_{\mu_{\boldsymbol{\xi}}}\left(N(K_{1,r},\G_{n})\right)}{\left(\E_{\mu_{\boldsymbol{\xi}}}\left[N(K_{1,r}, \G_n)\right]\right)^{2}} = o(1), \quad \text{ for } p\in (0,1) \text{ satisfying } n^{r+1}p^{r}\gg1,
	\end{equation}
	{where, for $\delta, \varepsilon>0$ $\tu{S}_{\delta,\varepsilon}$ is given by}
	\begin{equation*}
		\tu{S}_{\delta,\varepsilon} \coloneqq \left\{\boldsymbol{\xi}\in [0,1]^{n\choose 2}:\ \E_{\mu_{\boldsymbol{\xi}}}[N(K_{1,r}, \G_{n})] \geq \left(1+\delta(1+\varepsilon)\right)n^{r+1}p^{r}\right\}.
	\end{equation*}
	Postponing the proof of \eqref{eq:var-bd} to later let us first complete the proof of Theorem \ref{thm:mean-field}.  
	
	To this end, define $\widehat{\boldsymbol{\xi}} = (\widehat \xi_{i,j})_{i < j \in \set{n}} \in [0,1]^{n\choose 2}$ as follows, 
	\begin{equation}\label{eq:hat-xi}
		\widehat{\boldsymbol{\xi}}_{i,j} \coloneqq \begin{cases}
			p+\{\left(\delta(1-\varepsilon/2)\right)^{1/r}n^{1/r}p\}, \qquad &\text{if $i=1,j\in \set{n}\setminus\{1\}$},\\
			1, \qquad  &  i \in \set{\lfloor \delta (1-\varepsilon/2) n p^r \rfloor +1}\setminus \{1\}, j \notin \set{\lfloor \delta(1-\varepsilon/2) n p^r \rfloor +1}, \\
			p, &\text{otherwise}.
		\end{cases}
	\end{equation}
	It is easy to note that for large $n$, we have $\E_{\mu_{\widehat{\boldsymbol{\xi}}}}\left[N(K_{1,r}, \G_{n})\right] \geq \left(1+\delta(1-\varepsilon)\right)n^{r+1}p^{r}$.
	Therefore, by \eqref{eq:star loc regime 2 result}
	\begin{equation*}
		-\log \P\left(\UT\right) \geq I_{p}(\widehat{\boldsymbol{\xi}}) \geq \widehat\Psi_{p,r}\left(\delta(1-\varepsilon)\right),
	\end{equation*}
for all large $n$.

	To prove the other direction the broad strategy will be similar to that employed in the proof of \cite[Theorem~1.8]{AB}. Indeed, let ${\bm \xi}^{*}\in \tu{S}_{\delta,\varepsilon}$ satisfying $I_{p}({\bm \xi}^{*})\leq (1+\varepsilon/2)\widehat{\Psi}_{p,r}(\delta(1+\varepsilon))$. For convenience, we use the shorthand $f(\G_n) = N(K_{1,r},\G_n)$ and $\widehat{\Psi} = \widehat{\Psi}_{p,r}(\delta(1+\varepsilon))$. We define the function $\widehat{h}(\G_n)\coloneqq 2\widehat{\Psi}\cdot h(\widehat{f}(\G_n))$, where
	\[h(x) \coloneqq \begin{cases}
		2x+1 \quad &\text{for }x\leq -1,\\
		-x^2 &\text{for }x\in [-1,0],\\
		0 &\text{for }x\geq 0,
	\end{cases}\]
and 
\[\widehat{f}(\G_n)\coloneqq \begin{cases}
	\left(\frac{f(\G_n)}{(1+\delta(1+\varepsilon)\E_{\mu_p}[f])}-1\right)\cdot \left(\frac{\varepsilon\delta}{1+\delta(1+\varepsilon)}\right)^{-1}\qquad &\text{if}\quad \E_{\mu_{\bm \xi^{*}}}[f]\leq C(1+\delta)\E_{\mu_p}[f],\\
	\left(\frac{f(\G_n)}{\E_{\mu_{\bm \xi^{*}}}[f]}-1\right)(1-C^{-1})^{-1}&\text{otherwise},
\end{cases}\]
for some large absolute constant $C>1$.

It follows from the definition of $\widehat h$ that $\widehat{h}(\G_n)\leq -2\widehat{\Psi}$ on the event $\UT^{c}$. 
Further, the same argument as in \cite[p.~942]{AB} yields that $\E_{\mu_{\bm \xi^{*}}}[\widehat{h}(\G_n)]\geq -2\widehat{\Psi}\cdot \tu{Var}_{\mu_{\bm \xi^{*}}}(\widehat{f}(\G_n))$.
Equipped with the last two observations, 
using \eqref{eq:var-bd}, and an exponential change of measure, we repeat the computation in \cite[Eqn.~(5.17)]{AB} to obtain
\[
\P(\UT) \ge \exp(-(1+3\varepsilon/4)\widehat \Psi) - \exp(-2 \widehat \Psi),
\]
for all large $n$. Finally, noting that $\widehat\Psi\gtrsim p^{2}\log(1/p)$, (one may argue similarly as in   \cite[Lemma~5.2]{AB}), as $p \gtrsim n^{-2}$, the desired bound follows from above.

	We now turn to the proof of \eqref{eq:var-bd}.  Assume $\boldsymbol{\xi}\in \tu{S}_{\delta, \varepsilon}$. Let $\left(a_{i,j}\right)_{i,j\in \set{n}}$ be the random matrix which is the adjacency matrix of $\G_{n}$. 
	For $i<j\in \set{n}$, let $\G_{n}^{(i,j)}$ be the random graph on $n$ vertices obtained from $\G_{n}$ by replacing $a_{i,j}$ with $\widehat{a}_{i,j}$, an independent copy of $a_{i,j}$. By \cite[Theorem~3.1]{concineq} we have that 
	\begin{equation}\label{eq:es-ineq}
		\tu{Var}_{\mu_{\boldsymbol{\xi}}}\left(N(K_{1,r}, \G_{n})\right)\leq \frac{1}{2}\sum_{i<j}\E\left[\left(N(K_{1,r}, \G_n)-N(K_{1,r}, \G_{n}^{(i,j)})\right)^{2}\right].
	\end{equation}
	To upper bound the RHS of \eqref{eq:es-ineq} of we need some notation. For $i \ne j \in \set{n}$ and $k \in \N$ we let  $\T_i(k)$ and $\T_{i,j}(k)$ to be set of all pairwise distinct indices $\alpha_1, \alpha_2, \ldots, \alpha_k$ such that none of them are equal to $i$ and equal to either $i$ or $j$, respectively. 
	Now note that,
	\begin{equation*}
		N(K_{1,r}, \G_n)-N(K_{1,r}, \G_{n}^{(i,j)}) = (a_{i,j}-\widehat{a}_{i,j})\cdot\sum_{\alpha_{1}, \alpha_2,\ldots, \alpha_{r-1}\in \T_{i,j}(r-1)}\left(\prod_{l=1}^{r-1}a_{i,\alpha_{l}}+\prod_{l=1}^{r-1}a_{j,\alpha_{l}}\right).
	\end{equation*}
	Therefore, by \eqref{eq:es-ineq} 
\begin{equation}
	\label{mf: split ineq}
	\tu{Var}_{\mu_{\boldsymbol{\xi}}}\left(N(K_{1,r}, \G_{n})\right)   \lesssim \E_{\mu_{{\bm \xi}}}  N(K_{1,r}, \G_{n})+\sum_{k=r}^{2(r-1)}C_{k,r} \E_{\mu_{{\bm \xi}}} \left[\sum_{i=1}^{n}\sum_{\alpha_1, \ldots, \alpha_{k}\in \T_{i}(k)}\prod_{l=1}^{k}a_{i, \alpha_l}\right],
\end{equation}
where $C_{k,r}$ is the number of ways $k$ distinct elements can be split in two sets of size $r$. 
Next, define
\begin{equation*}
	Y_{k}\coloneqq \sum_{i=1}^{n}\sum_{\alpha_1, \ldots, \alpha_{k}\in \T_{i}(k)}\prod_{l=1}^{k}a_{i, \alpha_l}, \quad k\in \{r, \ldots, 2(r-1)\}.
\end{equation*} 
By our assumption $\boldsymbol{\xi}\in \tu{S}_{\delta, \varepsilon}$, we have $\E_{\mu_{\boldsymbol{\xi}}}N(K_{1,r}, \G_n)\gg 1$. By \eqref{mf: split ineq}, it remains to show for each $r\leq k\leq 2(r-1)$, $\E_{\mu_{\boldsymbol{\xi}}}\left[Y_k\right]\ll \E_{\mu_{\boldsymbol{\xi}}}\left(N(K_{1,r}, \G_n)\right)^{2}$. To this end, note that 
\begin{equation}
	\label{mf: max ineq}
	\E_{\mu_{\boldsymbol{\xi}}}[Y_k]\leq \sum_{i=1}^{n}\E_{\mu_{\boldsymbol{\xi}}}\left[\sum_{\alpha_{1}\ldots, \alpha_{r}\in \T_{i}(r)}\prod_{l=1}^{r}a_{i,\alpha_l}\right]\cdot \tu{D}_{i}^{k-r+1}\leq \widehat{\tu{D}}^{k-r+1}\cdot \E_{\mu_{\boldsymbol{\xi}}}\left[N(K_{1,r}, \G_n)\right],
\end{equation}
where $\tu{D}_{i}\coloneqq \sum_{j=1}^n\xi_{i,j}$ and $\widehat{\tu{D}}\coloneqq \max_{i\in \set{n}} \tu{D}_
i$. Now, if $\widehat{\tu{D}}\lesssim 1$, then we are done. So, assume $\widehat{\tu{D}}\gg 1$. We only need to show $\widehat{\tu{D}}^{r-1}\ll \E_{\mu_{\boldsymbol{\xi}}} \left[N(K_{1,r}, \G_n)\right]$. Let $i_*\in \set{n}$ be such that $\tu{D}_{i_*}=\widehat{\tu{D}}$. Then, observe that 
\begin{equation*}
	\E_{\mu_{\boldsymbol{\xi}}} \left[N(K_{1,r}, \G_n)\right]\geq \E_{\mu_{\boldsymbol{\xi}}}\left[\sum_{\alpha_{1}\ldots, \alpha_{r}\in \T_{i_{*}}(r)}\prod_{l=1}^{r}a_{i_{*},\alpha_l}\right] \geq \left(\tu{D}_{i_{*}}-r\right)^{r} \gg \widehat{\tu{D}}^{r-1},
\end{equation*}
where we have used that $\tu{D}_{i_{*}}\geq r$. This completes the proof.
\qed



\section{Proof of Theorem \ref{thm:poi}}

Let $\uH$ be a strictly balanced, connected graph with $e_{\uH}>0$. To make use of machinery from \cite{HMS}, we define $X_{\uH}$ to be the number of unlabelled copies of $\uH$ in $\G(n,p)$. Observe that $X_{\uH} = N(\uH, \G(n,p))/\tu{Aut}(\uH)$. Define $\mu_{\uH} \coloneqq \E[X_{\uH}]$. Let $\varepsilon>0$. Let $\eta = \eta(\varepsilon, \delta)$ be the constant from {\cite[Proposition~8.3]{HMS}} and let $K=K(\varepsilon, \delta, \eta)$ be the constant from {\cite[Lemma~8.5]{HMS}}. We note that $\mu_{\uH}$ satisfies \begin{equation*}
	\max\left\{\frac{1}{\eta}, K\right\}\leq \mu_{\uH} \leq \frac{\sqrt{M}}{K}, 
\end{equation*}
where $M= \text{Number of unlabelled copies of 
}\uH\text{ in }K_{n}$. We define a cluster of size $s$ to be a collection of $s$ many distinct copies of $\uH$ in $\G(n,p)$ such that each copy of $\uH$ in the cluster shares at least one edge with some other copy of $\uH$ in the cluster. Now, Theorem \ref{thm:poi} will follow from {\cite[Proposition~8.3]{HMS}} and {\cite[Lemma~8.5]{HMS}} once we show that for every $s$ satisfying $2\leq s\leq (\delta+\varepsilon)\mu,
$\begin{equation*}
	\E[D_s(X_{\uH})]\leq \exp(-Ks),
\end{equation*}
where $D_{s}(X_{\uH})$ is the number of clusters of size $s$ in $\G(n,p)$. Let $D_{s,k,m}$ be the number of clusters of size $s$ whose union (as a subgraph of $\G(n,p)$) has exactly $k$ vertices and $m$ edges. Similar to as in {\cite[Claim~8.7]{HMS}}, there exists a positive constant $\gamma$ such that for every $s\geq 2$, $k\geq v_{\uH}$, and $m\geq e_{\uH}+1$, we have
\begin{equation}
	\label{pr: ineq 1}
	\E[D_{s,k,m}]\leq n^{-2\gamma m}{k^{2} \choose m}{(2m)^{\alpha^{*}_{H}}\choose s}.
\end{equation}
To prove \eqref{pr: ineq 1}, we follow the same inductive reasoning as \cite{HMS}, with two modifications. First, we use the fact that for any proper subgraph $\tu{J}$ of $\uH$, with $e_{\tu{J}}>0$ satisfies,
\begin{equation*}
	e_{\uH}-e_{\tu{J}}\geq \lambda\left(1+\frac{1/\lambda^{*} - 1/\lambda}{v_{\uH}}\right)\cdot (v_{\uH}-v_{\tu{J}}),
\end{equation*}
where $\lambda = e_{\uH}/v_{\uH}$ and $\lambda^{*} = \max_{\phi\neq \tu{I}\subsetneq H} e_{\tu{I}}/{v_{\tu{I}}}$. The above holds because $\uH$ is strictly balanced. Secondly, we use {\cite[Theorem~5.7]{HMS}} to show that the number of unlabelled copies $\uH$ in a graph with $k$ vertices and $m$ edges is at most $(2m)^{\alpha^{*}_{\uH}}$. The remainder of the proof is straightforward and follows the same line of argument as \cite{HMS}. So, we omit the details.
\qed

\bibliographystyle{plain}
\bibliography{mybib.bib}{}

\appendix
\section{Proof of Theorem \ref{thm: cond struc regular graph}}\label{sec:proof-regular}
We begin with a couple of combinatorial results.
\begin{lemma}
\label{cond_struc:bipartite lemma regular}
Assume $p\ll1$. Fix $\varepsilon>0$. Let $\tu{H}$ be a $\Delta$-regular connected graph and $\tu{G}\subseteq K_{n}$ be such that $e(\tu{G})\lesssim n^{2}p^{\Delta}$. There exists a constant $\eta\coloneqq \eta(\varepsilon, H)>0$, such that with $U= U(\eta)\coloneqq\{v\in V(\tu{G}):\tu{deg}_{\tu{G}}(v)\geq \eta n\}$ and $V=V(\eta)\coloneqq V(\tu{G})\setminus U$, we have
\begin{equation}
	\label{Q_H structural}
	N_{U}(\tu{J},\tu{G}[U,V])\geq N(\tu{J}, \tu{G})-\varepsilon n^{v_{\tu{J}}}p^{e_{\tu{J}}}, \quad \forall \tu{J}\in Q_{\uH}\setminus\{ {\uH}\}
\end{equation}
and, setting $\pi_{\tu{H}} \coloneqq \mathds{1}\{\tu{H is bipartite}\}$,
\begin{equation}
	\label{H structural}
	N(\uH, \tu{G}[V]) + \bigg(N_{U}(\tu{H}, \tu{G}[U,V];A)+N_{U}(\tu{H}, \tu{G}[U,V]; A^{c})\bigg)\pi_{\tu{H}}\geq N(\uH, \tu{G})-\varepsilon n^{v_{\uH}}p^{e_{\uH}},
\end{equation}
where for a bipartite $\uH$ is with bipartition $V(\uH)=A\cup A^{c}$ and $B \in \{A, A^c\}$ we let $N_{U}(H, \tu{G}[U,V]; B)$ 
to be the number of labelled copies of $H$ in $\tu{G[U,V]}$ such that $B$ 
is mapped to $U$.
\end{lemma}
We also require the following stability result. 

\begin{lemma}
\label{stability:regular}
Let $\uH$ be a $\Delta$-regular connected graph. Suppose $\tu{G}$ is a graph which satisfes 
\begin{equation*}
	N(\uH, \tu{G})\geq (1-\varepsilon)\cdot (2e(\tu{G}))^{v_{\uH}/2}, 
\end{equation*}  
for some $\varepsilon\geq e(\tu{G})^{-1/2}$, then $\tu{G}$ has a subgraph $\widehat{\tu{G}}$ with minimum degree at least $(1-4\varepsilon^{1/2})\cdot (2e(\tu{G}))^{1/2}$. 
\end{lemma}
Lemma \ref{stability:regular} extends \cite[Theorem 5.11]{HMS} which derived the result for $\tu{H}=K_r$, the clique on $r$ vertices.
Taking Lemmas \ref{cond_struc:bipartite lemma regular} and \ref{stability:regular} as given, we proceed to prove Theorem \ref{thm: cond struc regular graph}(a).

\begin{proof}[Proof of Theorem \ref{thm: cond struc regular graph}(a)]
For any connected, $\Delta$-regular graph $\uH$ and $p\in (0,1)$ satisfying $n^{-1/\Delta}\ll p\ll 1$,  
\eqref{eq:entropic stability HMS} was established in \cite[Section~7]{HMS}. Further in this regime, for any fixed  $\delta>0$, \cite{BGLZ} shows that 
\begin{equation}
	\label{limit variational problem:regular}
	\lim_{n\rightarrow\infty} \frac{\Phi_{\uH}(\delta)}{n^{2}p^{\Delta}\log(1/p)} = \min\{\theta_{\uH}(\delta), \delta^{2/v_{\uH}}/2\} \eqqcolon \widetilde{\theta}_{\tu{H}}(\delta).
\end{equation}
Therefore  by \eqref{cond_struc:step 1} 
and the continuity of the map $\delta\mapsto \widetilde{\theta}_{\tu{H}}(\delta)$ , we have in the regime $n^{-1/\Delta}\ll p\ll 1$
\begin{equation*}
	\P\left(\big\{\exists \tu{G}\subseteq K_{n}:\ \tu{G}\in \mathcal{I}^{*}\tu{ and }e(\tu{G})\leq (1+\mathfrak{s}_{0}(\varepsilon))\widetilde{\theta}_{\tu{H}}(\delta)n^{2}p^{\Delta}\bigg| \UTH\right)\geq 1-(\P(\UTH))^{\varepsilon/16},
\end{equation*}
for some non-negative function $\mathfrak{s}_{0}(\cdot) = \mathfrak{s}_{0}(\cdot, \uH, \delta)$ such that $\lim_{\varepsilon\downarrow 0}\mathfrak{s}_{0}(\varepsilon) = 0$.  

For any graph $\tu{G}\in \mathcal{I}^{*}$ satisfying $e(\tu{G})\leq (1+\mathfrak{s}_{0}(\varepsilon))\widetilde{\theta}_{\tu{H}}(\delta) n^{2}p^{\Delta}$, we will prove that $\tu{G}$ must contain either an almost-clique or an almost-complete bipartite subgraph of appropriate size. This will complete the proof. 

To this end, we begin by noting that \eqref{eq:J-sum} continues to hold when $\uH$ is regular. Further, by \cite[Lemma 5.2]{HMS} the bound \eqref{s1:2nd eq} continues to hold in this case. Therefore, by Lemma \ref{cond_struc:bipartite lemma regular} for any $\tu{G}\in \mathcal{I}^{*}$ and $e(\tu{G})\leq (1+\mathfrak{s}_{0}(\varepsilon)) \widetilde{\theta}_{\uH}(\delta)n^2 p^\Delta$ we have 
\begin{equation*}
	\sum_{\tu{J}\in Q_{\uH}\setminus\{\uH\}}\frac{N_{U}(\tu{J}, \tu{G}[U,V])}{n^{v_{\tu{J}}}p^{e_{\tu{J}}}}+\frac{N_{U}(\uH,\tu{G}[U,V], A)\cdot \pi_{\uH}}{n^{v_{\uH}}p^{e_{\uH}}} +\frac{N_{U}(\uH,\tu{G}[U,V], A^{c})\cdot \pi_{\uH}}{n^{v_{\uH}}p^{e_{\uH}}} + \frac{N(\uH, \tu{G}[V])}{n^{v_{\uH}}p^{e_{\uH}}}\geq \delta -4\varepsilon.
\end{equation*} 
Consequently there exists an $x\in [0,1]$ such that
\begin{equation}
	\label{H in G[V]}
	N(\uH,\tu{G}[V])\geq (1-x)\cdot(\delta-4\varepsilon)\cdot n^{v_{\uH}}p^{e_{\uH}}
\end{equation}
and 
\begin{equation}
	\label{J in G[U,V]}
	\sum_{\tu{J}\in Q_{\uH}\setminus\{\uH\}}\frac{N_{U}(\tu{J}, \tu{G}[U,V])}{n^{v_{\tu{J}}}p^{e_{\tu{J}}}}+\frac{N_{U}(\uH,\tu{G}[U,V], A)\cdot \pi_{\uH}}{n^{v_{\uH}}p^{e_{\uH}}} +\frac{N_{U}(\uH,\tu{G}[U,V], A^{c})\cdot \pi_{\uH}}{n^{v_{\uH}}p^{e_{\uH}}} \geq x\cdot(\delta-4\varepsilon).
\end{equation}
Further using arguments similar to the ones employed in \eqref{cond_struc:step 4} we derive from \eqref{J in G[U,V]} that
\begin{equation}
	\label{J in G[U,V]-new}
	P_{\uH}\left(\frac{e(G[U,V])}{n^{2}p^{\Delta}} \right) \geq 1+ x\cdot(\delta-4\varepsilon).
\end{equation}
Next by \cite[Theorem~5.4]{HMS} and the fact that $P_{\uH}$ is strictly increasing and continuous, we deduce from above
\begin{equation}
	\label{edge lower bound}
	\begin{cases}
		e(\tu{G}[V])\geq \left((1-x)^{2/v_{\uH}}\delta^{2/v_{\uH}}/2-\mathfrak{s}_{1}(\varepsilon)\right)\cdot n^{2}p^{\Delta},\\
		e(\tu{G}[U,V])\geq \left(\theta_{\uH}(\delta x)-\mathfrak{s}_{1}(\varepsilon)\right) \cdot n^{2}p^{\Delta},
	\end{cases}
\end{equation}
where 
$\mathfrak{s}_{1}(\cdot) \coloneqq \mathfrak{s}_{1}(\cdot,\uH, \delta)$ is some non-negative function such that $\lim_{\varepsilon\downarrow0} \mathfrak{s}_{1}(\varepsilon) = 0$.
Since the map $x\mapsto \theta_{\uH}(\delta x)+\left((1-x)\cdot \delta\right)^{2/v_{\uH}}/2$ is continuous and strictly concave in $[0,1]$, as $e(\tu{G}) \leq (1+\mathfrak{s}_0(\varepsilon))\widetilde \theta_{\uH}(\delta)n^2p^\Delta$, we may choose  $\varepsilon$ sufficiently small so that  \eqref{edge lower bound} holds for $x \in [0, \mathfrak{s}_{2}(\varepsilon)] \cup [1-\mathfrak{s}_{2}(\varepsilon),1]$ for some non-negative $\mathfrak{s}_2(\cdot)$ with $\lim_{\varepsilon \downarrow 0} \mathfrak{s}_2(\varepsilon)=0$. Hence, enlarging $\mathfrak{s}_1(\cdot)$ we may and will assume that  \eqref{edge lower bound}  holds for some $x=x^{*}\in \{0,1\}$. Using this lower bound, observing that $e(\tu{G}[V])+e(\tu{G}[U,V])\leq e(\tu{G})$, and the upper bound on $e(\tu{G})$ we obtain the following tight upper bounds:
\begin{equation}
	\begin{cases}
		e(\tu{G}[V])\leq \left((1-x^{*})^{2/v_{\uH}}\delta^{2/v_{\uH}}/2+\mathfrak{s}_{0}(\varepsilon)+\mathfrak{s}_{1}(\varepsilon)\right)\cdot n^{2}p^{\Delta},\\
		e(\tu{G}[U,V])\leq \left(\theta_{\uH}(\delta x^{*})+\mathfrak{s}_{0}(\varepsilon)+\mathfrak{s}_{1}(\varepsilon)\right) \cdot n^{2}p^{\Delta}.
	\end{cases}
\end{equation} 
If $x^{*}=1$, we proceed as in proof of Theorem \ref{cond_struc:loc-1} and conclude that $W\subseteq V(\tu{G})$ such that
\begin{equation*}
	\min_{w\in W}\tu{deg}_{\tu{G}}(w)\geq (1-\widetilde{\mathfrak{t}}(\varepsilon))n\quad \text{and}\quad e(\tu{G}[W, V(\tu{G})\setminus W])\geq (\theta_{\uH}-\mathfrak{t}(\varepsilon))n^{2}p^{\Delta},
\end{equation*}
for some non-negative $\mathfrak{t}(\varepsilon)$ and $\widetilde{\mathfrak{t}}(\varepsilon)$ such that $\lim_{\varepsilon\downarrow 0}\max\{\mathfrak{t}(\varepsilon), \widetilde{\mathfrak{t}}(\varepsilon)\}=0$. On the other hand if $x^{*}=0$, Lemma \ref{stability:regular} yields a subset $\widehat{V}$ of $V$ such that 
\begin{equation*}
	|\widehat{V}|\geq (1-\mathfrak{g}(\varepsilon))\delta^{1/v_{\uH}} \cdot np^{\Delta/2}
	\quad \text{and}\quad \min_{v\in \widehat{V}}\tu{deg}_{\tu{G}[\widehat{V}]}(v)\geq (1-\widetilde{\mathfrak{g}}(\varepsilon))|\widehat{V}|,
\end{equation*}
for some non-negative $\mathfrak{g}(\varepsilon)$ and $\widetilde{\mathfrak{g}}(\varepsilon)$ such that $\lim_{\varepsilon\downarrow 0}\max\{\mathfrak{g}(\varepsilon), \widetilde{\mathfrak{g}}(\varepsilon)\}=0$. This yields \eqref{eq:clique-hub}. 

To obtain \eqref{eq:clique-or-hub} we simply note that there exists a unique $\delta_0(\uH) >0$ such that $\theta_{\uH}(\delta) < \delta^{2/v_{\uH}}/2$ if $\delta < \delta_0(\uH)$ and $\theta_{\uH}(\delta) > \delta^{2/v_{\uH}}/2$ if $\delta > \delta_0(\uH)$ (cf.~\cite[Eqn.~(1.4)]{BGLZ}). Hence, $x^*=1$ if $\delta < \delta_0(\uH)$ and $x^*=0$ if $\delta > \delta_0(\uH)$. This yields \eqref{eq:clique-or-hub}. 
\end{proof}

We now address the proofs of Lemmas \ref{cond_struc:bipartite lemma regular} and \ref{stability:regular}. 

\begin{proof}[Proof of Lemma \ref{cond_struc:bipartite lemma regular}]
Let $\tu{G}\subseteq K_{n}$ satisfy $e(\tu{G})\leq Cn^{2}p^{\Delta}$ for some constant $C<\infty$. 
Let $\widetilde \eta<1$ be as in Lemma \ref{cond_struc: bipartite lemma}. By Remark \ref{rmk:J-copy-bound}  the bound \eqref{Q_H structural} holds for any $\tu{J}\in Q_{\uH}\setminus\{\uH\}$ and any $\eta \in (0, \widetilde \eta]$. 

Next suppose $\uH$ is non-bipartite. Fix $a\in V(\uH)$ and let $\uH_{a}$ denote the subgraph obtained by removing $a$  and all edges incident to it. By \cite[Lemma~5.2]{HMS} we have $\alpha_{\uH_{a}}^{*}\leq v_{\uH}/2$. We claim $\alpha_{\uH_{a}}^{*}< v_{\uH}/2$. If not, then  $\uH_{a}\in Q_{\uH}$ (again by \cite[Lemma 5.2]{HMS}). In particular, $\uH_a$ is bipartite. However, since all neighbors of $a$ in $\uH$ has degree  strictly less than $\Delta$ in $\uH_a$, reintroducing $a$ and its adjacent edges would force $\uH$ to be bipartite - a contradiction.  Using this strict inequality, we bound the number of labelled copies of $\uH$ where at least one vertex of $\uH$ is mapped to $U = U(\widetilde \eta)$,
\begin{equation*}
	N(\uH, \tu{G})-N(\uH, \tu{G}[V])\leq \sum_{a\in V(\uH)} |U|\cdot N(\uH_{a}, \tu{G})\leq \sum_{a\in V(\uH)}|U|\cdot (2e(\tu{G}))^{v_{\uH_{a}}-\alpha_{\uH_{a}}^{*}}\cdot n^{2\alpha_{\uH_{a}}^{*}-v_{\uH_{a}}}\leq \varepsilon n^{v_{\uH}}p^{e_{\uH}},
\end{equation*}
where in the last step we used $p\ll1$. 

We now turn to prove \eqref{H structural} when $\uH$ is bipartite. Let $\widetilde\Phi$ be the number of labelled copies of $\uH$ in $\tu{G}$ where a pair of adjacent vertices of $\uH$ is mapped to $U = U(\widetilde \eta)$ and $\widetilde \Psi$ be the number of labelled copies of $\uH$ where two neighbors of a vertex in $\uH$ is mapped to $U$ and $V=V(\widetilde \eta)$, respectively. Arguing similarly as in \eqref{bound cond Phi} we find that $\widetilde \Phi\leq (\varepsilon/2) n^{v_{\uH}}p^{e_{\uH}}$. It remains to prove the same for $\widetilde\Psi$. 

To this end, fix $a\in V(\uH)$ and let $b$ and $c$ be two distinct neighbors of $a$. Let $\widetilde\Psi_{a,b,c}$ be the number of labelled copies of $\uH$ in $\tu{G}$ such that $b$ is mapped to $U$ and $c$ is mapped to $V$.  Consider a $2$-matching $M$ formed by the union of two disjoint perfect matchings which contain the edges $(a,b)$ and $(a,c)$ respectively. $M$ consists only of disjoint even cycles, say $C_l, C_{l_1}, \cdots, C_{l_{k}}$, with $C_{l}$ containing $\{a,b,c\}$ for some $l =l(\{a,b,c\})$ determined by the triplet $\{a,b,c\}$. For $W, \widetilde{W}\subset V(\tu{G})$  we define $\widetilde\psi_{a,b,c}(W,\widetilde{W})$ to be the number of labelled copies of $C_{l}$ in $\tu{G}$ such that $b$ is mapped to $W$ and $c$ is mapped to $\widetilde{W}$. Thus,
\begin{equation}\label{eq:pre-choice of eta}
	\widetilde \Psi_{a,b,c} \leq \widetilde\psi_{a,b,c}(U,U^{c}) \cdot   \prod_{i=1}^{k}N(C_{l_{i}}, \tu{G})\leq \widetilde\psi_{a,b,c}(U,U^{c})  \cdot (2C n^{2}p^{\Delta})^{(v_{\uH}-l_{1})/2}.
\end{equation}
We will show that there exists a choice of 
$\eta<\widetilde{\eta}$ such that for any $a \in V(\uH)$ and any choice of neighbors $b$ and $c$ of $a$, with $l=l(\{a,b,c\})$ as above,
\begin{equation}
	\label{choice of eta}
	\widetilde \psi_{a,b,c}(U({\eta}),U({\eta})^{c})\leq \varepsilon^2 n^{l_{}}p^{(\Delta l_{})/2}.
\end{equation} 
This together with \eqref{eq:pre-choice of eta} will yield the desired bound for $\widetilde \Psi$.
Define $s^{*}\coloneqq \lceil 2\Delta^2 v_{\uH} (2C)^{v_{\uH}/2}/\varepsilon^2\rceil $. For $i\in \set{s^{*}} \cup \{0\}$
define $U_{i} \coloneqq \{v\in V(\tu{G}):\tu{deg}_{\tu{G}}(v)\geq \widetilde\eta^{2(i+1)}\}$.  

Let $\mathcal{S}_{\{a,b,c\}} $ be set of indices $i \in \set{s^*-1}\cup\{0\}$ such that  $\widetilde \psi_{a,b,c}(U_{i}, U_{i+1}\setminus U_{i}) > (\varepsilon^2/2)n^{l}p^{\Delta l/2}$. We claim that 
\begin{equation}\label{eq:cS-bd}
	|\mathcal{S}_{\{a,b,c\}}| <s^*/\Delta^2 v_{\uH}. 
\end{equation} 
Otherwise, as $\{U_{i+1}\setminus U_i\}_{i \in \set{s^*-1} \cup \{0\}}$ are disjoint,
\begin{equation*}
	(2C)^{v_{\uH}} n^l p^{\Delta l/2} < \sum_{i \in \mathcal{S}_{\{a,b,c\}}} \widetilde \psi_{a,b,c}(U_{i}, U_{i+1}\setminus U_{i})\leq N(C_{l}, \tu{G})\leq (2e(\tu{G}))^{l/2}\leq (2C)^{l/2}n^{l}p^{\Delta l/2},
\end{equation*}
a contradiction. On the other hand, noting the fact that $C_{l}$ can be covered by a  disjoint union of $P_4$ containing $\{a,b,c\}$ with $b$ as a leaf and a matching of size $(l-4)/2$, we have the bound 
\begin{equation*}
	\widetilde \psi_{a,b,c}(U_{i}, U_{i+1}^{c})\leq |U_{i}|\cdot 2e(\tu{G})\cdot \widetilde \eta^{2(i+2)}n \cdot (2e(\tu{G}))^{(l-4)/2} \leq \frac{\varepsilon^2}{2} \cdot n^{l}p^{\Delta l/2}.  
\end{equation*}  
To obtain the first inequality above we have used that the number of choices for the leaf of $P_4$  that is to mapped to some vertex in $U_i$ is at most $|U_i|$ and the number of copies of $K_{1,2}$ with central vertex belonging to $U_{i+1}^c$ is at most $2 e(\tu{G}) \cdot \widetilde \eta^{2(i+2)} n$. The rightmost inequality is a consequence of the definition of $\widetilde \eta$ (see \eqref{eta:defn}). 

Thus we have shown that for any triplet $\{a, b, c\}$ and any $i \notin \mathcal{S}_{\{a,b,c\}}$ the inequality \eqref{choice of eta} holds with $U(\eta)$ replaced by $U_i$. Finally, by \eqref{eq:cS-bd} we deduce that $\set{s^*-1} \cup \{0\}\setminus \cup_{\{a,b,c\}} \mathcal{S}_{\{a,b,c\}} \neq \phi$. Hence, there indeed exists some $\eta \in (0, \widetilde{\eta}]$ such that \eqref{choice of eta} holds for all triplets $\{a,b,c\}$. This completes the proof.
%
\end{proof}

\begin{proof}[Proof of Lemma \ref{stability:regular}]
It follows from \cite[Lemma~5.1]{HMS} that the vertices of any $\Delta$-regular connected graph $\uH$ can be covered by a collection $\mathcal{C}$ of vertex disjoint edges and cycles. If  $\mathcal{C}$ contains no odd cycles,  we may replace each cycle in $\mathcal{C}$ with one of its perfect matching, thereby transforming $\mathcal{C}$ into a perfect matching of $\uH$. Now, observe that any additional edge of $\uH$ not in $\mathcal{C}$ is adjacent to exactly two edges of $\mathcal{C}$. Consequently, $\uH$ can be covered by a subgraph that is a disjoint union of $P_{4}$ 
and a matching of size $(v_{\uH}-4)/2$. Thus,
\begin{equation}
	(1-\varepsilon)\cdot (2e(\tu{G}))^{v_{\uH}/2}\leq N(\uH,\tu{G})\leq N(P_{4},\tu{G})\cdot (2e(\tu{G}))^{(v_{\uH}-4)/2}.
\end{equation} 
The desired result now follows directly from  \cite[Claim~5.12]{HMS}. 

Next assume $\mathcal{C}$ contains an odd cycle $C_{l}$ for some $l\geq 3$. We then have the following inequality
\begin{equation*}
	(1-\varepsilon)\cdot  (2e(\tu{G}))^{v_{\uH}/2}\leq N(\uH, \tu{G})\leq N(C_{l}, \tu{G})\cdot \prod_{\tu{J}\in \mathcal{C}\setminus\{C_l\}} N(\tu{J}, \tu{G}) \leq N(C_{l}, \tu{G})\cdot (2e(\tu{G}))^{(v_{\uH}-l)/2}.
\end{equation*} 
This reduction shows that it is enough to establish Lemma \ref{stability:regular} for $C_l$. Since the case $l=3$ is already covered in \cite[Claim 5.13]{HMS} we may and will assume that $l \geq 5$. Next, for each edge $e\in E(\tu{G})$, let $c_{e}$ be the number of unlabelled copies of $C_{l}$ in $\tu{G}$ that contain the edge $e$. The proof of \cite[Lemma~5.5]{HMS} establishes the following key inequality:
\begin{equation*}
	N(C_{l}, \tu{G})^{2}\leq 2e(\tu{G})\cdot \sum_{e\in E(\tu{G})}2c_{e}^{2}\leq 2e(\tu{G})\cdot \Lambda(C_{l}^{*}),
\end{equation*}
where $C_{l}^{*}$ is the graph obtained by gluing two copies of $C_{l}$ along an edge and $\Lambda(C_{l}^{*})$ is the number of homomorphisms of $C_{l}^{*}$ in $\tu{G}$ where one copy of $C_{l}$ in $C_{l}^{*}$ is mapped to distinct vertices. Now, $\Lambda(C_{l}^{*})$ is upper bounded by $N(P_{4},\tu{G})\cdot (2e(\tu{G}))^{l-3}$ as $C_{l}^{*}\setminus P_{4}$ is covered by a matching of size $l-3$, and as $\l \ge 5$ we have $P_4 \subseteq C^{*}_{l}$. The result now follows from \cite[Claim~5.12]{HMS}.
\end{proof}

It remains to prove Theorem \ref{thm: cond struc regular graph}(b). We split the proof into two parts.
\begin{proof}[Proof of Theorem \ref{thm: cond struc regular graph}(b) for $p \gtrsim n^{-1/\Delta} (\log n)^{-v_{\uH}}$] In this regime from \cite{BGLZ} it follows that 
\[
\lim_{n \to \infty} \frac{\Phi_{\uH}(\delta)}{n^2 p^\Delta \log(1/p)}=\frac12 \delta^{2/v_{\uH}}.
\]
Therefore, as $p \ll n^{-1/\Delta}$, by \cite[Lemma 5.2 and Theorem 5.4]{HMS} for any $\tu{G} \in \mathcal{J}^*$ and any $\phi \neq \tu{J} \subsetneq \uH$ without any isolated vertices we find that $N(\tu{J}, \tu{G}) = o(n^{v_{\tu{J}}} p^{e_{\tu{J}}})$. Hence, by \eqref{eq:J-sum}, for any $\tu{G} \in \mathcal{J}^*$ we have $N(\uH, \tu{G}) \geq (\delta -3 \varepsilon) n^{v_{\uH}} p^{e_{\uH}}$, for all large $n$. Since the bound \eqref{eq:entropic stability HMS} follows from \cite[Proposition 7.1]{HMS} the proof now follows upon applying Theorem \ref{HMS:technical result} and Lemmas \ref{cond_struc:bipartite lemma regular} and \ref{stability:regular}.
\end{proof}

Similar to the proof of Theorem \ref{cond_struc:star loc-2} it requires notions of core and strong-core graphs. We modify Definitions \ref{def:c-graph} and \ref{def:sc-graph} to suit our current need. These are borrowed from \cite{ABRB}.

\begin{definition}
Let $\overline C \coloneqq \overline C(v_{\uH}, \Delta, \delta, \varepsilon) < \infty$ be some sufficiently large constant.		A graph $\tu{G}\subseteq K_n$ is said to be a core graph if the following hold:
\begin{itemize}
	\item[(C1)]
	$N(\uH, \tu{G})\geq \delta(1-3\varepsilon)n^{v_{\uH}}p^{e_{\uH}}$,
	\item[(C2)] 
	$e(\tu{G})\leq \cbar n^{2}p^\Delta \log(1/p)$, 
\end{itemize}
and 
\begin{itemize}
	\item[(C3)] $\min_{e\in E(\tu{G})}N(\uH, \tu{G},e)\geq \delta \varepsilon n^{v_{\uH}}p^{e_{\uH}}/ \left(\cbar n^{2}p^\Delta \log(1/p)\right)$.
\end{itemize}
\end{definition}

\begin{definition}
Let $\cstar\coloneqq 32\delta^{2/v_{\uH}}$. 
A graph $\tu{G}\subseteq K_{n}$ is said to be a \textit{strong-core} graph if the following hold:
\begin{itemize}
	\item[(SC1)] $N(\uH, \tu{G})\geq \delta(1-6\varepsilon) n^{v_{\uH}}p^{e_{\uH}}$,
	\item[(SC2)]$e(\tu{G})\leq \cstar n^{2}p^\Delta$,
\end{itemize}
and
\begin{itemize}
	\item[(SC3)]$\min_{e\in E(\tu{G})} N(\uH, \tu{G}, e)\geq (\delta \varepsilon/\cstar)\cdot (np^{\Delta/2})^{v_{\uH}-2}$. 
\end{itemize}
\end{definition}

\begin{proof}[Proof of Theorem \ref{thm: cond struc regular graph}(b) for $p \ll n^{-1/\Delta} (\log n)^{-v_{\uH}}$]
We begin by claiming that
\begin{equation}\label{eq:event-Omega}
	\left\{ \exists \tu{G} \subseteq \G(n,p): \tu{G} \text{ is a core graph} \right\}\setminus \left\{ \exists \tu{G} \subseteq \G(n,p): \tu{G} \text{ is a strong-core graph} \right\} \subseteq \Omega,
\end{equation}
for some event $\Omega$ such that 
\begin{equation}\label{eq:prob-Omega}
	\P(\Omega) \leq (\P(\UTH))^2.
\end{equation}
This claim is immediate from the arguments used in \cite[Sections 3.1--3.3]{ABRB}. In particular, for $p$ such that $n^2p^\Delta \geq (\log n)^{2 v_{\uH}}$ the assertions \eqref{eq:event-Omega}-\eqref{eq:prob-Omega} are direct consequences of \cite[Lemma 3.5]{ABRB}. On the other hand, for $p$ below that threshold the claim \eqref{eq:event-Omega} follows from \cite[Eqn.~(3.22)]{ABRB} and the argument employed in \cite[p.~32]{ABRB}. The claimed probability bound \eqref{eq:prob-Omega} is immediate from \cite[Lemmas 3.7, 3.8, 3.10, and 3.11]{ABRB}. Next, as $\uH$ is non-bipartite it follows from the proof of \cite[Proposition 3.3]{ABRB} (see \cite[Section 4.1]{ABRB}) that 
\begin{equation}\label{eq:sc-many-edge}
	\P\left(\exists \tu{G} \subseteq \G(n,p): \tu{G} \text{ is a strong-core graph with } e(\tu{G}) \geq (1+C \varepsilon)\delta^{2/v_{\uH}}n^2p^\Delta/2 \right) \leq (\P(\UTH))^{1+\varepsilon},
\end{equation}
for some constant $C < \infty$ (not depending on $\varepsilon$). Further, the proof of \cite[Lemma 3.2]{ABRB} (similar to strengthening mentioned in Remark \ref{rmk:HMS-improve}) yields that
\begin{equation}\label{eq:no-core}
	\P\left(\left\{\G(n,p) \text{ contains a core graph}\right\}^c \cap \UTH \right) \leq (\P(\UTH))^2. 
\end{equation}
Therefore, applying \eqref{eq:event-Omega}-\eqref{eq:no-core} we derive
\begin{multline*}
	\P\left(\left\{ \G(n,p) \text{ contains a strong-core graph } \tu{G}  \text{ with } e(\tu{G}) \geq (1+C \varepsilon)\delta^{2/v_{\uH}}n^2p^\Delta/2\right\}\Big| \UTH \right) \\
	\geq 1 - (\P(\UTH))^{\varepsilon/2}.
\end{multline*}
The proof now completes upon applying Lemmas \ref{cond_struc:bipartite lemma regular} and \ref{stability:regular}.
\end{proof}

\end{document}